\documentclass[a4paper,reqno,10pt]{amsart}
\usepackage{amsfonts}
\usepackage{amsmath}
\usepackage{amssymb}
\usepackage{tikz-cd}
\usepackage[sort]{cite}
\usepackage{bbm}
\usepackage{amscd}
\usepackage[left=2cm,right=2cm,bottom=3cm,top=3cm]{geometry}
\usepackage{etoolbox}
\usepackage{filecontents}
\usepackage{xcolor}
\usepackage{hyperref}

\hypersetup{
  colorlinks=true,
  linkcolor=blue,
  citecolor=red
}
\newtheorem{theorem}{Theorem}[section]
\newtheorem{proposition}[theorem]{Proposition}
\theoremstyle{definition}
\newtheorem{lemma}[theorem]{Lemma}
\newtheorem{claim*}[theorem]{claim*}
\newtheorem{definition}[theorem]{Definition}
\newtheorem{corollary}[theorem]{Corollary}

\newtheorem{remark}[theorem]{Remark}

\newtheorem{example}[theorem]{Example}

\begin{document}
\title[An infinitary Gowers--Hales--Jewett Ramsey theorem]{Actions of trees
on semigroups, and\\
an infinitary Gowers--Hales--Jewett Ramsey theorem}
\author[Martino Lupini]{Martino Lupini}
\address{Martino Lupini, Mathematics Department, California Institute of
Technology, 1200 East California Boulevard, Mail Code 253-37, Pasadena, CA
91125}
\email{lupini@caltech.edu}
\urladdr{http:n//www.lupini.org/}
\thanks{The author was supported by the NSF Grant DMS-1600186.}
\dedicatory{}

\begin{abstract}
We introduce the notion of (Ramsey) action of a tree on a (filtered)
semigroup. We then prove in this setting a general result providing a common
generalization of the infinitary Gowers Ramsey theorem for multiple tetris
operations, the infinitary Hales--Jewett theorems (for both located and
nonlocated words), and the Farah--Hindman--McLeod Ramsey theorem for layered
actions on partial semigroups. We also establish a polynomial version of our
main result, recovering the polynomial Milliken--Taylor theorem of
Bergelson--Hindman--Williams as a particular case. We present applications
of our Ramsey-theoretic results to the structure of delta sets in amenable
groups.
\end{abstract}

\maketitle

\section{Introduction\label{Section:introduction}}

The finitary \emph{Hales--Jewett theorem} \cite{hales_regularity_1963} is a
fundamental combinatorial pigeonhole principle. Several years after the
original proof of Hales and Jewett, two deep infinitary strengthenings of
the Hales--Jewett theorem (for located and nonlocated words) have been
proved in \cite{bergelson_partition_1994} using the theory of ultrafilters
and algebra in the Stone-\v{C}ech compactification. In another direction,
and around the same time, Gowers established in \cite{gowers_lipschitz_1992}
another fundamental combinatorial pigeonhole principle, which has been since
then referred to the (infinite) \emph{Gowers Ramsey theorem}. Gowers' Ramsey
theorem is a far-reaching generalization of \emph{Hindman's theorem on
finite unions} \cite{hindman_finite_1974}. We refer to \cite%
{gowers_ramsey_2003,kanellopoulos_proof_2004,todorcevic_introduction_2010,tyros_primitive_2015,ojeda-aristizabal_finite_2015}
for other proofs of such a result and its finitary counterpart, where
explicit bounds on the quantities involved are also obtained.

A common generalization of Gowers' Ramsey theorem and the infinite
Hales--Jewett theorems has been established by Farah--Hindman--McLeod in the
setting of \emph{layered actions }on \emph{adequate partial semigroups }\cite%
{farah_partition_2002}. In a different direction, the infinite Gowers Ramsey
theorem has been strengthened in \cite{lupini_gowers_2016} by considering 
\emph{multiple tetris operations}. This answered a question of Barto\v{s}ov%
\'{a} and Kwiatkowska from \cite{bartosova_gowers_2014}, where the
corresponding finitary statement is proved. A common generalization of
Gowers' Ramsey theorem for multiple tetris operations and the \emph{%
Milliken--Taylor theorem }\cite{milliken_ramseys_1975,taylor_canonical_1976}
is also provided in \cite{lupini_gowers_2016}.

Gowers' Ramsey theorem for multiple tetris operation does not fit in the
framework of layered actions on partial semigroups developed by
Farah--Hindman--McLeod in \cite{farah_partition_2002}. It is therefore
natural to wonder whether there exists a unifying combinatorial principle
that lies at the heart of both Gowers' Ramsey theorem for multiple tetris
operations and the infinite Hales--Jewett theorems, as well as the
Farah--Hindman--McLeod Ramsey theorem for layered actions on partial
semigroups. The goal of the present paper is to provide such a unifying
combinatorial principle within the framework, here introduced, of \emph{\
Ramsey actions of rooted trees on filtered semigroups}. Our main result is
Theorem \ref{Theorem:product-tree-filtered}, which provides a common
generalization of all the results mentioned above. One can also obtain from
such a general result more direct common generalizations of Gowers' theorem
for multiple tetris operations and the Hales--Jewett theorems (for located
and nonlocated words). Such common generalizations---Theorem \ref%
{Theorem:Gowers-Hales-Jewett} and Theorem \ref%
{Theorem:nonlocated-Gowers-Hales-Jewett}---are stated in terms of variable
words with \emph{variables indexed by a finite rooted tree}, and variable
substitution maps that respect the tree structure. We also provide a common
generalization of our main result---Theorem \ref%
{Theorem:product-tree-filtered}---and the \emph{polynomial Milliken--Taylor
theorem} of Bergelson--Hindman--Williams \cite[Corollary 3.5]%
{bergelson_polynomial_2014}; see Theorem \ref{Theorem:polynomial-Gowers}.
All the results of this paper are infinitary, and imply by a routine
compactness argument their finitary counterparts. We omit the statement of
these finitary counterparts, leaving it to the interested reader. We will
conclude by presenting applications of some of our Ramsey-theoretic results
to the structure of delta sets in amenable graphs.

The present paper consists of six sections, besides this introduction. In
Section \ref{Section:action-compact} we introduce and study the notion of
action of an ordered set and of a rooted tree on a compact right topological
semigroup. Section \ref{Section:action-partial} deals with the notion of
(Ramsey) action of an ordered set and of a rooted tree on a partial
semigroup. General result for Ramsey actions of rooted trees on adequate
partial semigroups is obtained here (Theorem \ref%
{Theorem:monochromatic-action} and Corollary \ref%
{Corollary:product-action-tree}). Section \ref{Section:examples} explains
how Gowers' theorem for multiple tetris operations and the Hales--Jewett
theorem are both subsumed by Theorem \ref{Theorem:monochromatic-action}.
Section \ref{Section:action-filtered} considers the even more general
framework of (Ramsey) actions of rooted trees on filtered semigroups. It is
explained here how all the previous results extend to this more general
framework. This allows one to recover the infinite Hales--Jewett theorem for
nonlocated words. Section \ref{Section:polynomial} presents a further
polynomial generalization, which subsumes this main result of the paper as
well as the polynomial Milliken--Taylor theorem \cite[Corollary 3.5]%
{bergelson_polynomial_2014}. Finally, Section \ref{Section:combinatorial}
presents applications to combinatorial configurations contained in delta
sets is amenable groups.

After the present paper was written, we have been informed that a general
Ramsey statement subsuming Gowers' theorem for multiple tetris operations
and the infinitary Hales--Jewett theorems has  been independently obtained by Solecki
with different methods. We refer the reader to \cite{solecki_general_2016}
for this alternative approach.

In the rest of this paper we denote by $\omega $ be the set of natural
numbers including $0$, and $\mathbb{N}$ be the set of natural numbers
different from zero. We identify an element $n$ of $\omega $ with the set $%
\left\{ 0,1,\ldots ,n-1\right\} $ of its predecessors. If $A,B$ are finite
nonempty subsets of $\omega $, we write $A<B$ if the maximum element of $A$
is smaller than the minimum element of $B$. We also write $A<\ell $ for $%
A\subset \omega $ and $\ell \in \omega $ if the largest element of $A$ is
smaller than $\ell $. Given a set $A$ we let $\left[ A\right] ^{<\aleph
_{0}} $ be the set of finite subsets of $A$. If $D$ is a set, then we denote
by $\beta D$ the space of ultrafilters on $D$; see \cite[Chapter 3]%
{hindman_algebra_2012}. This is endowed with a canonical compact Hausdorff
topology, having the sets $\overline{A}=\left\{ \mathcal{U}\in \beta D:A\in 
\mathcal{U}\right\} $ for $A\subset D$ as basis of open (and closed) sets.
We will use in the rest of the paper the notation of \emph{ultrafilter
quantifiers}; see \cite[Chapter 1]{todorcevic_introduction_2010}. If $\psi
\left( x\right) $ is a formula depending on a variable $x$ ranging over $D$,
then we write $\left( \mathcal{U}x\right) $ $\varphi \left( x\right) $ as an
abbreviation for $\left\{ x\in D:\varphi \left( x\right) \text{ holds}%
\right\} \in D$. In particular, we have that $\left( \mathcal{U}x\right) $ $%
x\in A$ is equivalent to the assertion that $A\in \mathcal{U}$. By a \emph{%
finite coloring }of the set $D$ we mean a function $c:D\rightarrow n$ for
some $n\in \omega $. Any such a coloring admits a \emph{canonical extension}%
, which we still denote by $c$, to a finite coloring of $\beta D$, obtained
by setting $c\left( \mathcal{U}\right) =i$ if and only if $\left( \mathcal{U}%
x\right) $, $c\left( x\right) =i$.

\subsubsection*{Acknowledgments} We would like to thank Andy Zucker for his comments on a first draft of the present paper.

\section{Actions of trees on compact right topological semigroups\label%
{Section:action-compact}}

\subsection{Compact right topological semigroups\label%
{Subsection:compact-semigroup}}

We recall here some notions concerning compact right topological semigroups.
An (additively denoted) compact right topological semigroup $X$ is a
semigroup $\left( X,+\right) $ endowed with a compact topology with the
property that, for every $z\in X$, the right translation map $x\mapsto x+z$
is continuous. \emph{In the following we assume all the compact right
topological semigroup to be Hausdorff}. An element $e$ of $X$ is \emph{%
idempotent }if $e+e=e$. A classical result of Ellis \cite[Corollary 2.10]%
{ellis_lectures_1969}---see also \cite[Lemma 2.1]%
{todorcevic_introduction_2010}---asserts that any compact right topological
semigroup contains idempotent elements. One can define an order among
idempotents of $X$ by setting $e_{0}\leq e_{1}$ if and only if $%
e_{0}+e_{1}=e_{1}+e_{0}=e_{0}$. An idempotent element of $X$ is \emph{%
minimal }if it is minimal with respect to such an order. The proof of \cite[%
Corollary 2.10]{ellis_lectures_1969} also shows that for any idempotent
element $e_{0}$ of $X$ there exists a minimal idempotent $e_{0}$ of $X$ such
that $e\leq e_{0}$.

A closed subsemigroup $A$ of $X$ is a \emph{nonempty }closed subset of $X$
with the property that $x+y\in A$ whenever $x,y\in A$. Observe that the
idempotent elements of $X$ are precisely the closed subsemigroups of $X$
that contain a single element. A closed subsemigroup $A$ of $X$ is a closed 
\emph{bilateral ideal} if $x+a$ and $a+x$ belong to $A$ whenever $a\in A$
and $x\in X$. We denote by $\mathcal{S}\left( X\right) $ the set of closed
subsemigroups of $X$. We define an order in $\mathcal{S}\left( X\right) $ by
setting $A\leq B$ if and only if $\left( A+B\right) \cup \left( B+A\right)
\subset A$. Clearly a subsemigroup $A$ of $X$ is a bilateral ideal if and
only if $A\leq X$. Observe that such an order extend the order on
idempotents defined above, when an idempotent element $e$ of $X$ is
identified with the closed subsemigroup $\left\{ e\right\} $. If $X$ is a
compact right topological semigroup, we define $\mathrm{\mathrm{End}}\left(
X\right) $ to be the set of continuous semigroup homomorphisms $\tau
:X\rightarrow X$. Observe that $\mathrm{End}\left( X\right) $ is a semigroup
with respect to composition.

In the following we will regard $\mathcal{S}\left( X\right) $ as an ordered
set endowed with such an ordering. (Here and in the following, all the
ordered sets are supposed to be \emph{partially }ordered.) We record here
for future reference the following well known fact; see also \cite[Lemma 2.3]%
{todorcevic_introduction_2010}.

\begin{lemma}
\label{Lemma:minimal-lift}Suppose that $X\ $is a compact right topological
semigroup. If $A,B\in \mathcal{S}\left( X\right) $ and $A\leq B$, then for
any idempotent $b\in B$ there exists a minimal idempotent $a$ of $A$ such
that $a\leq b$.
\end{lemma}

\begin{proof}
Consider a minimal idempotent element $a$ of $A+b$. Observe that $b+a$ is an
idempotent element of $b+a$ such that $\left( b+a\right) +a=b+a$. Therefore
by minimality of $a$ inside $A+b$ we have that $b+a=a=a+b$ and hence $a\leq
b $. Suppose now that $z$ is an idempotent element of $A$ such that $z+a=z$.
Then we have that $z+b=z+a+b=z+a=z$ and hence $z\in A+b$. It follows from
minimality of $a$ inside $A+b$ that $z=a$. Therefore $a$ is a minimal
idempotent element of $A$.
\end{proof}

\subsection{Actions of trees on compact right topological semigroups\label%
{Subsection:action-compact-semigroup}}

Suppose that $\mathbb{P}$ is an ordered set, and $X$ is a compact right
topological semigroup.

\begin{definition}
\label{Definition:action-compact}An \emph{action} $\alpha $ of $\mathbb{P}\ $%
on $X$ is given by

\begin{itemize}
\item an order-preserving function $\mathbb{P}\rightarrow \mathcal{S}\left(
X\right) $, $t\mapsto X_{t}$,

\item a subsemigroup $\mathcal{F}_{\alpha }\subset \mathrm{End}\left(
X\right) $,
\end{itemize}

such that for every $\tau \in \mathcal{F}_{\alpha }$ there exists an
function $f_{\tau }:\mathbb{P}\rightarrow \mathbb{P}$---which we call the 
\emph{spine} of $\tau $---such that $\tau $ maps $X_{t}$ to $X_{f_{\tau
}\left( t\right) }$ for every $t\in \mathbb{P}$, and such that $\tau \left(
x\right) =x$ for any $x\in X_{t}$ and $t\in \mathbb{P}$ such that $f_{\tau
}\left( t\right) =t$.
\end{definition}

Given an action $\alpha $ of $\mathbb{P}$ on $X$ we let $X_{\alpha }$ be set
of functions $\xi :\mathbb{P}\rightarrow X$ such that $\xi \left( t\right)
\in X_{t}$ and $\tau \circ \xi =\xi \circ f_{\tau }$ for every $\tau \in 
\mathcal{F}_{\alpha }$ and $t\in \mathbb{P}$. When $X_{\alpha }$ is
nonempty, we endow $X_{\alpha }$ with the product topology and the entrywise
operation. This turns $X_{\alpha }$ into a compact right topological
semigroup. Observe that an idempotent in $X_{\alpha }$ is an element $\xi $
of $X_{\alpha }$ such that $\xi \left( t\right) $ is an idempotent element
of $X_{t}$ for every $t\in \mathbb{P}$. We say that an idempotent $\xi $ in $%
X_{\alpha }$ is \emph{order-preserving} if $\xi \left( t_{0}\right) \leq \xi
\left( t_{1}\right) $ whenever $t_{0}\leq t_{1}$.

Suppose now that $T$ is a rooted tree. We regard $T$ as an ordered set
endowed with the canonical rooted tree order obtained by setting $t^{\prime
}\leq t$ if and only if $t^{\prime }$ is a descendent of $t$.

\begin{definition}
\label{Definition:regrssive}A\emph{\ regressive homomorphism }of $T$ is a
function $f:T\rightarrow T$ such that $f\left( t\right) \geq t$ for every $%
t\in T$, and $f$ maps two adjacent nodes either to the same node or to
adjacent nodes.
\end{definition}

It is clear that any regressive homomorphism fixes the root, and maps every
branch to itself.

\begin{definition}
\label{Definition:rooted-action-compact}A \emph{Ramsey action} $\alpha $ of $%
T\ $on $X$ is given by an action of $T$ on $X$ in the sense of Definition %
\ref{Definition:action-compact} such that $X_{\alpha }$ is nonempty and, for
every $\tau \in \mathcal{F}_{\alpha }$, the corresponding spine $f_{\tau
}:T\rightarrow T$ is a regressive homomorphism.
\end{definition}

A similar proof as \cite[Lemma 2.1]{lupini_gowers_2016} shows the following.

\begin{proposition}
\label{Proposition:tree}Suppose that $T$ is a rooted tree of height $\leq
\omega $ with root $r$. If $\alpha $ is a Ramsey action of $T$ on $X$, then $%
X_{\alpha }$ contains an order-preserving idempotent. Furthermore, if $\xi
^{(0)}$ is an idempotent element of $X_{\alpha }$, then $X_{\alpha }$
contains an order-preserving idempotent $\xi $ such that $\xi \left(
r\right) =\xi ^{(0)}\left( r\right) $.
\end{proposition}

\begin{proof}
Fix an idempotent element $\xi ^{(0)}$ of $X_{\alpha }$. Let, for $k\in
\omega $, $\pi _{k}:T\rightarrow T$ be the function that maps every node to
its $k$-th predecessor, where we convene that the $k$-th predecessor of a
node of height at most $k$ is the root, and the $0$-th predecessor of every
node is itself. Let $T_{k}$ be the set of nodes in $T$ of height at most $k$%
. We define by recursion on $k\in \omega $ idempotent elements $\xi ^{(k)}$
of $X_{\alpha }$ such that $\xi ^{(k)}\left( t\right) +\xi ^{(k)}\left(
t_{0}\right) =\xi ^{(k)}\left( t\right) $ whenever $t_{0}\in T_{k-1}$, and $%
t\in T$ are such that $t\leq t_{0}$, and $\xi ^{(k)}\left( t\right) =\xi
^{(j)}\left( t\right) $ whenever $j\leq k$ and $t\in T_{j}$. Granted the
construction one can then consider $\xi \in X_{\alpha }$ defined by%
\begin{equation*}
\xi \left( t\right) :=(\xi ^{(n)}\circ \pi _{n}+\xi ^{(n)}\circ \pi
_{n-1}+\cdots +\xi ^{(n)}\circ \pi _{0})\left( t\right)
\end{equation*}%
for any node $t\in T$ of height $n$. It is not difficult to verify that $\xi
\in X_{\alpha }$ is an order-preserving idempotent such that $\xi \left(
t\right) =\xi ^{(0)}\left( r\right) $.

We proceed now with the recursive construction. We have already defined $\xi
^{(0)}$. Suppose that $\xi ^{(0)},\ldots ,\xi ^{(k)}$ have been defined for
some $k\in \omega $. Consider the closed subsemigroup $Z_{k}$ of $\xi \in
X_{\alpha }$ such that $\xi \left( t_{0}\right) =\xi ^{(k)}\left(
t_{0}\right) $ for $t\in T_{k}$, and $\xi \left( t\right) +\xi \left(
t_{0}\right) =\xi \left( t\right) $ for every $t\in T$ and $t_{0}\in T_{k}$
such that $t\leq t_{0}$. Observe that $Z_{k}$ is nonempty. Indeed, set%
\begin{equation*}
\xi \left( t\right) :=(\xi ^{(k)}\circ \pi _{0}+\xi ^{(k)}\circ \pi
_{1}+\cdots +\xi ^{(k)}\circ \pi _{n})\left( t\right)
\end{equation*}%
for any node $t\in T$ of height $n$. Observe that $\tau \circ \xi =\xi \circ
f_{\tau }$ for every $\tau \in \mathcal{F}_{\alpha }$ since $\tau $ is a
homomorphism, $f_{\tau }$ is an regressive homomorphism of $T$, and $\tau
\circ \xi ^{(k)}=\xi ^{(k)}\circ f_{\tau }$ by recursive assumption.
Therefore $\xi \in X_{\alpha }$. Furthermore if $t_{0}\in T_{k}$, then $\xi
^{(k)}\left( t\right) +\xi ^{(k)}\left( \pi ^{(j)}\left( t\right) \right)
=\xi ^{(k)}\left( t\right) $ for every $j\in \omega $ and hence $\xi \left(
t\right) =\xi ^{(k)}\left( t\right) $. Finally suppose that $t_{0}\in T_{k}$
and $t\in T$ are such that $t\leq t_{0}$.\ We want to prove that $\xi \left(
t\right) +\xi \left( t_{0}\right) =\xi \left( t\right) $. Suppose that the
height of $t$ is $m$. If $m\leq k$ then we have that%
\begin{equation*}
\xi \left( t\right) +\xi \left( t_{0}\right) =\xi ^{(k)}\left( t\right) +\xi
^{(k)}\left( t_{0}\right) =\xi ^{(k)}\left( t\right) =\xi \left( t\right)
\end{equation*}%
by the recursive assumption. Suppose now that $m>k$. Then we have that, by
the recursive assumption,%
\begin{eqnarray*}
\xi \left( t\right) +\xi \left( t_{0}\right) &=&\xi ^{(k)}\left( t\right)
+\cdots +\xi ^{(k)}\left( \pi _{m-k}\left( t\right) \right) +\xi
^{(k)}\left( t_{0}\right) \\
&=&\xi ^{(k)}\left( t\right) +\cdots +\xi ^{(k)}\left( \pi _{m-k}\left(
t\right) \right) \\
&=&\xi \left( t\right) \text{.}
\end{eqnarray*}%
This concludes the proof that $\xi \in Z_{k}$. Since $Z_{k}$ is nonempty, it
contains an idempotent element $\xi ^{(k+1)}\in Z_{k}$. This concludes the
recursive construction.
\end{proof}

The following definition is inspired by the definition of layered action
from \cite[Definition 3.3]{farah_partition_2002}.

\begin{definition}
\label{Definition:layered}An action $\alpha $ of $T\ $on $X$ is a \emph{%
layered action} if for every $\tau \in \mathcal{F}_{\alpha }$ and $t\in T$
one has that

\begin{enumerate}
\item $f_{\tau }$ is equal to either $t$ or the immediate predecessor of $t$;

\item if $t$ has an immediate predecessor $t^{-}$, then for any minimal
idempotent $p\in X_{t^{-}}$ there exists $q\in X_{t}$ such that $\sigma
\left( q\right) =p$ for any $\sigma \in \mathcal{F}_{\alpha }$ such that $%
f_{\sigma }\left( t\right) =t^{-}$.
\end{enumerate}
\end{definition}

A similar proof as \cite[Theorem 3.8]{farah_partition_2002} gives the
following.

\begin{proposition}
\label{Proposition:layered}Suppose that $T$ is a rooted tree of height $\leq
\omega $, and $\alpha $ is a layered action of $T$. Then $\alpha $ is a
Ramsey action. Furthermore, $X_{\alpha }$ contains an order-preserving
idempotent $\xi $ such that $\xi \left( t\right) $ is a minimal idempotent
in $X_{t}$ for every $t\in T$.
\end{proposition}

\begin{proof}
It is clear by definition of layered action that, for every $\tau \in 
\mathcal{F}_{\alpha }$, $f_{\tau }$ is a regressive homomorphism. We now
prove the second assertion. This will also show that $\alpha $ is a Ramsey
action.

We define minimal idempotents $x_{t}\in X_{t}$ by recursion on the height of 
$t$ such that, for every $t,t^{\prime }\in T$ and $\tau \in \mathcal{F}%
_{\alpha }$, $\tau \left( x_{t}\right) =x_{f_{\tau }\left( t\right) }$ and $%
x_{t^{\prime }}\leq x_{t}$ if $t^{\prime }\leq t$. If $t_{0}$ is the root of 
$T$ then we let $x_{t_{0}}$ be any minimal idempotent element of $X_{t_{0}}$%
. Suppose that $x_{t}$ has been defined whenever the height of $t$ is at
most $k$. Suppose now that $t$ has height $k+1$ and let $t^{-}$ be the
immediate predecessor of $t$.

If, for each $\tau \in \mathcal{F}_{\alpha }$, $f_{\tau }\left( t\right) =t$%
, one can just define $x_{t}\in X_{t}$ using Lemma \ref{Lemma:minimal-lift}.
Suppose now that $f_{\tau }\left( t\right) =t^{-}$ for some $\tau \in 
\mathcal{F}_{\alpha }$. Let $Y$ be the set of $z\in X_{t}$ such that $\tau
\left( z\right) =x_{t^{-}}$ for every $\tau \in \mathcal{F}_{\alpha }$ such
that $f_{\tau }\left( t\right) =t^{-}$. By hypothesis we have that $Y$ is
nonempty. Observe now that $Y+x_{t^{-}}\subset Y$. Indeed we have that for $%
y\in Y$, $\tau \left( y+x_{t^{-}}\right) =x_{t^{-}}+x_{f_{\tau }\left(
t^{-}\right) }=x_{t^{-}}$ by recursive hypothesis. Pick now a minimal
idempotent $x_{t}$ of $Y+x_{t^{-}}$. Observe that $x_{t^{-}}+x_{t}\in
Y+x_{t^{-}}$ is an idempotent such that $x_{t^{-}}+x_{t}\leq x_{t}$. By
minimality, $x_{t^{-}}+x_{t}=x_{t}=x_{t}+x_{t^{-}}$ and hence $x_{t}\leq
x_{t^{-}}$. Observe that if $z\in Y$ is an idempotent element such that $%
z\leq x_{t}$ then $z\in Y+x_{t^{-}}$. This shows that $x_{t}$ is minimal in $%
Y$. Finally suppose that $z\in X_{t}$ is an idempotent element such that $%
z\leq x_{t}$. Then we have that, for any $\tau \in \mathcal{F}_{\alpha }$, $%
\tau \left( z\right) $ is an idempotent element of $X_{t^{-}}$ such that $%
\tau \left( z\right) \leq x_{t^{-}}$. It follows by minimality of $x_{t^{-}}$
that $\tau \left( z\right) =x_{t^{-}}$, and hence $z\in Y$. Minimality of $%
x_{t}$ in $Y$ now shows that $z=x_{t}$. This concludes the proof that $x_{t}$
is minimal in $X_{t}$.
\end{proof}

\section{Actions of trees on partial semigroups\label{Section:action-partial}%
}

\subsection{Partial semigroups\label{Subsection:partial-semigroup}}

A \emph{partial semigroup }\cite[Definition 1.2]{farah_partition_2002}---see
also \cite[Section 2]{bergelson_partition_1994} and \cite[Section 2.2]%
{todorcevic_introduction_2010}---is a set $S$ endowed with a partially
defined binary operation $\left( x,y\right) \mapsto x+y$, with the property
that $\left( x+y\right) +z=x+\left( y+z\right) $ for $x,y,z\in S$. Such an
equality should be interpreted as asserting that the left hand side is
defined if and only if the right hand side is defined, and in such a case
they are equal. A partial semigroup is \emph{adequate }\cite[Definition 2.1]%
{farah_partition_2002}---or \emph{directed }\cite[\S 2.2]%
{todorcevic_introduction_2010}---if for every finite subset $A$ of $S$ the
set $\varphi _{S}(A)$ of elements $x$ of $S$ such that $a+x$ is defined for
every $a\in A$ is nonempty.

A \emph{partial semigroup homomorphism} \cite[Definition 2.8]%
{farah_partition_2002} between partial semigroups $S$ and $T$ is a function $%
f:S\rightarrow T$ with the property that $f\left( x+y\right) =f\left(
x\right) +f\left( y\right) $ for $x,y\in S$. Again, such an equality should
be interpreted as asserting that the left hand side is defined if and only
if the right hand side is defined, and in such a case they are equal. A
partial semigroup homomorphism is\emph{\ adequate }if for every finite
subset $B$ of $T$ there exists a finite subset $A$ of $S$ such that the
image of $\varphi _{S}(A)$ under $f$ is contained in $\varphi _{T}\left(
B\right) $. If $S\subset T$ then we say that $S$ is an\emph{\ adequate
partial subsemigroup} of $T$ is the inclusion map is an adequate partial
semigroup homomorphism \cite[Definition 2.10]{farah_partition_2002}. We say
furthermore that $S$ is an\emph{\ adequate bilateral ideal} if it is an
adequate partial subsemigroup, and, for every $x\in S$ and $y\in T$, $%
x+y,y+x $ belong to $S$ whenever they are defined.

If $S_{0},S_{1}$ adequate partial subsemigroups of $S$, then we let $%
S_{0}\leq S_{1}$ if $\left( S_{0}+S_{1}\right) \cup \left(
S_{1}+S_{0}\right) \subset S_{0}$. This should be interpreted as the
assertion that, for any $s_{0}\in S_{0}$ and $s_{1}\in S_{1}$, $s_{0}+s_{1}$
and $s_{1}+s_{0}$ belong to $S_{0}$ whenever they are defined. Observe that $%
S_{0}\leq S$ if and only if $S_{0}$ is an adequate bilateral ideal of $S$.
We denote by $\mathcal{S}\left( S\right) $ the space of adequate partial
subsemigroups of $S$. We regard $\mathcal{S}\left( S\right) $ as an ordered
set with respect to the ordering just defined.

\subsection{Cofinal ultrafilters on partial semigroups\label%
{Subsection:cofinal-ultrafilter}}

Suppose that $S$ is a partial semigroup. An ultrafilter $\mathcal{U}$ over $%
S $ is \emph{cofinal }if $\forall x\in S$, $\mathcal{U}y$, $x+y$ is defined.
Following \cite[Chapter 2]{todorcevic_introduction_2010}, we denote by $%
\gamma S$ the space of cofinal ultrafilters over $S$. It is clear that $%
\gamma S\ $is a closed subspace of the space of ultrafilters over $S$.
Furthermore, $\gamma S$ is a compact right topological semigroup when
endowed with the operation defined by setting $A\in \mathcal{U}+\mathcal{V}$
if and only if $\mathcal{U}x$, $\mathcal{V}y$, $x+y\in A$; see \cite[%
Corollary 2.7]{todorcevic_introduction_2010} and \cite[Theorem 2.6]%
{farah_partition_2002}. More generally, this expression defines a function $%
\beta S\times \gamma S\rightarrow \beta S$, $\left( \mathcal{U},\mathcal{V}%
\right) \rightarrow \mathcal{U}+\mathcal{V}$ such that, for any $\mathcal{V}%
\in \gamma S$, the function $\beta S\rightarrow \beta S$, $\mathcal{U}%
\mapsto \mathcal{U}+\mathcal{V}$ is continuous. In particular, for any $s\in
S$ and $\mathcal{U}\in \gamma S$, the element $s+\mathcal{U}$ of $\beta S$
is well defined.

Suppose that $S_{0}$ and $S_{1}$ are partial semigroups, and $\sigma
:S_{0}\rightarrow S_{1}$ is an adequate partial semigroup homomorphism. Then 
$\sigma $ induces a continuous semigroup homomorphism $\sigma :\gamma
S_{0}\rightarrow \gamma S_{1}$ by setting $A\in \sigma \left( \mathcal{U}%
\right) $ if and only if $\mathcal{U}x$, $\sigma \left( x\right) \in A$ \cite%
[Lemma 2.14]{farah_partition_2002}. When $S_{0}$ is an adequate subsemigroup
of $S_{1}$ and $\sigma :S_{0}\rightarrow S_{1}$ is the inclusion map, the
continuous extension $\sigma :\gamma S_{0}\rightarrow \gamma S_{1}$ is
one-to-one. In this situation, we will identify $\gamma S_{0}$ with its
image under $\sigma $, which is the closed subsemigroup of $\gamma S_{1}$
consisting of the cofinite ultrafilters on $S_{1}$ that contain $S_{0}$.
This defines a map $\mathcal{S}\left( S\right) \rightarrow \mathcal{S}\left(
\gamma S\right) $, $S_{0}\mapsto \gamma S_{0}$. Here $\mathcal{S}\left(
\gamma S\right) $ denotes as in Subsection \ref{Subsection:compact-semigroup}
the space of closed subsemigroups of $\gamma S$. It is not hard to see that
such a map is order-preserving with respect to the ordering on $\mathcal{S}%
\left( S\right) $ and $\mathcal{S}\left( \gamma S\right) $ defined above.

\subsection{Actions of ordered sets on partial semigroups\label%
{Subsection:action-partial-semigroup}}

Suppose that $\mathbb{P}$ is an ordered set, and $S$ is an adequate partial
semigroup. We denote by $\mathrm{End}\left( S\right) $ the space of \emph{%
adequate }partial semigroup homomorphisms $\tau :S\rightarrow S$. Observe
that $\mathrm{\mathrm{En}d}\left( S\right) $ is a semigroup with respect to
composition.

\begin{definition}
\label{Definition:action-partial}An \emph{action }$\alpha $ of $\mathbb{P}$
on $S$ is given by

\begin{itemize}
\item an order-preserving function $\mathbb{P}\rightarrow \mathcal{S}\left(
S\right) $, $t\mapsto S_{t}$, and

\item a subsemigroup $\mathcal{F}_{\alpha }\subset \mathrm{End}\left(
S\right) $,
\end{itemize}

such that such that for every $\tau \in \mathcal{F}_{\alpha }$ there exists
a function $f_{\tau }:\mathbb{P}\rightarrow \mathbb{P}$---which we call the 
\emph{spine} of $\tau $---such that $\tau $ maps $S_{t}$ to $S_{f_{\tau
}\left( t\right) }$ for every $t\in \mathbb{P}$, and such that $\tau \left(
s\right) =s$ for any $s\in S_{t}$ and $t\in T$ such that $f_{\tau }\left(
t\right) =t$.
\end{definition}

Suppose that $\alpha $ is an action of $\mathbb{P}$ on $S$. Then $\alpha $
induces an action in the sense of Definition \ref{Definition:action-compact}
of $\mathbb{P}$ on the compact right topological semigroup $X=\gamma S$,
which we still denote by $\alpha $. This is obtained by setting $%
X_{t}:=\gamma S_{t}$ for $t\in T$ and considering the semigroup of
continuous semigroup homomorphisms $\tau :\gamma S\rightarrow \gamma S$
obtained as the canonical continuous extensions of elements $\tau $ of $%
\mathcal{F}_{\alpha }$. Consistently with the notation introduced in
Subsection \ref{Subsection:action-compact-semigroup}, we denote by $(\gamma
S)_{\alpha }$ the set of functions $\xi :\mathbb{P}\rightarrow \gamma S$
such that $\xi \left( t\right) \in \gamma S_{t}$ and $\tau \circ \xi =\xi
\circ f_{\tau }$ for every $t\in \mathbb{P}$. An order-preserving idempotent
in $\left( \gamma S\right) _{\alpha }$ is an element $\xi $ of $\left(
\gamma S\right) _{\alpha }$ such that $\xi \left( t\right) $ is an
idempotent element in $\gamma S_{t}$ and $\xi \left( t\right) \leq \xi
\left( t_{0}\right) $ whenever $t,t_{0}\in \mathbb{P}$ are $t\leq t_{0}$.

\begin{theorem}
\label{Theorem:monochromatic-action}Suppose that $\alpha $ is an action of a
finite ordered set $\mathbb{P}$ on the adequate partial semigroup $S$.
Suppose that $\xi \in \left( \gamma S\right) _{\alpha }$ is an
order-preserving idempotent. Fix a finite coloring $c$ of $S$ and consider
its canonical extension to a finite coloring $c$ of $\beta S$. Fix a
sequence $(\psi _{n}^{(\mathcal{F})})$ of functions $\psi _{n}^{(\mathcal{F}%
)}:(S^{\mathbb{P}})^{n}\rightarrow \left[ \mathcal{F}_{\alpha }\right]
^{<\aleph _{0}}$ and a sequence $(\psi _{n}^{(S)})$ of functions $\psi
_{n}^{(S)}:(S^{\mathbb{P}})^{n}\rightarrow \left[ S\right] ^{<\aleph _{0}}$
such that $\psi _{n}\left( x_{0},\ldots ,x_{n-1}\right) $ contains the range
of $\tau _{i}\circ x_{i}$ for every $i\in n$ and $\tau _{i}\in \psi _{i}^{(%
\mathcal{F})}\left( x_{0},\ldots ,x_{i-1}\right) $. There exists a sequence $%
\left( x_{n}\right) $ of functions $x_{n}:\mathbb{P}\rightarrow S$ such that

\begin{itemize}
\item $x_{n}\left( t\right) \in S_{t}\cap (\varphi _{S}\circ \psi
_{n}^{(S)})\left( x_{0},\ldots ,x_{n-1}\right) $ for every $n\in \omega $
and $t\in \mathbb{P}$; and

\item for any $\ell \in \omega $, $n_{0}<n<\cdots <n_{\ell }\in \omega $, $%
t_{i}\in \mathbb{P}$ for $i\leq \ell $, and $\tau _{i}\in \psi _{n_{i}}^{(%
\mathcal{F})}\left( x_{0},\ldots ,x_{n_{i}-1}\right) $ for $i\leq \ell $, if 
$\left\{ f_{\tau _{i}}\left( t_{i}\right) :i\leq \ell \right\} $ is a chain
in $\mathbb{P}$ with least element $t$, then the color of $\tau _{0}\left(
x_{n_{0}}\left( t_{0}\right) \right) +\cdots +\tau _{\ell }\left( x_{n_{\ell
}}\left( t_{\ell }\right) \right) $ is equal to the color of $\xi \left(
t\right) $.
\end{itemize}
\end{theorem}

\begin{proof}
We now define by recursion on $m\in \omega $ functions $x_{m}:\mathbb{P}%
\rightarrow S$ such that $x_{m}\left( t\right) \in S_{t}\cap (\varphi
_{S}\circ \psi _{m}^{(S)})\left( x_{0},\ldots ,x_{m-1}\right) $ such that
for every $m\in \omega $ the following holds:

\begin{enumerate}
\item[($1_{m}$)] for every $\ell \leq m$, $n_{0}<n_{1}<\cdots <n_{\ell }\leq
m$, $t_{i}\in \mathbb{P}$ for $i\leq \ell $, $\tau _{i}\in \psi _{i}^{(%
\mathcal{F})}\left( x_{0},\ldots ,x_{n_{i}-1}\right) $ for $i\leq \ell $
such that $\left\{ f_{\tau _{i}}\left( t_{i}\right) :i\leq \ell \right\} $
is a chain in $\mathbb{P}$ with least element $t$, one has that the color of 
$\tau _{0}\left( x_{n_{0}}\left( t_{0}\right) \right) +\cdots +\tau _{\ell
}\left( x_{n_{\ell }}\left( t_{\ell }\right) \right) $ is equal to the color
of $\xi \left( t\right) $, and

\item[($2_{m}$)] for every $\ell \leq m$, $n_{0}<n_{1}<\cdots <n_{\ell }\leq
m$, $t_{i}\in \mathbb{P}$ for $i\leq \ell +1$, $\tau _{i}\in \psi _{n_{i}}^{(%
\mathcal{F})}\left( x_{0},\ldots ,x_{n_{i}-1}\right) $ for $i\leq \ell $
such that $\left\{ f_{\tau _{i}}\left( t_{i}\right) :i\leq \ell \right\}
\cup \left\{ t_{\ell +1}\right\} $ is a chain in $\mathbb{P}$ one has that
the color of $\tau _{0}\left( x_{n_{0}}\left( t_{0}\right) \right) +\cdots
+\tau _{\ell }\left( x_{n_{\ell }}\left( t_{\ell }\right) \right) +\xi
\left( t_{\ell +1}\right) $ is equal to the color of $\xi \left( t\right) $.
\end{enumerate}

Let us consider initially the case $m=0$. In this case $\left( S^{\mathbb{P}%
}\right) ^{0}$ is a single point. Therefore $\psi _{0}^{(S\mathfrak{)}}$
selects a finite subset $S_{0}$ of $S$, and $\psi _{0}^{(\mathcal{F})}$
selects a finite subset $\mathcal{F}_{0}$ of $\mathcal{F}_{\alpha }$. We
need to find a function $x_{0}:\mathbb{P}\rightarrow S$ such that $%
x_{0}\left( t\right) \in S_{t}\cap \varphi _{S}\left( S_{0}\right) $ for
every $t\in \mathbb{P}$ and such that the following holds:

\begin{enumerate}
\item[($1_{0}$)] for every $t_{0}\in \mathbb{P}$ and $\tau _{0}\in \mathcal{F%
}_{0}$ the color of $\tau _{0}\left( x_{0}\left( t_{0}\right) \right) $ is
equal to the color of $\xi \left( f_{\tau _{0}}\left( t_{0}\right) \right) $%
, and

\item[($2_{0}$)] for every $t_{0},t_{1}\in \mathbb{P}$ and $\tau \in 
\mathcal{F}_{0}$ if $\left\{ f_{\tau _{0}}\left( t_{0}\right) ,t_{1}\right\} 
$ is a chain in $\mathbb{P}$ with least element $t$, then the color of $\tau
_{0}\left( x_{0}\left( t_{0}\right) \right) +\xi \left( t_{1}\right) $ is
equal to the color of $\xi \left( t\right) $.
\end{enumerate}

Fix $t\in \mathbb{P}$. Using the notation of ultrafilter quantifiers for the
ultrafilter $\xi \left( t\right) $, we have that $\xi \left( t\right) s$, $%
\forall \tau _{0}\in \mathcal{F}_{0}$, $\forall t_{1}\in \mathbb{P}$ such
that $\left\{ f_{\tau _{0}}\left( t\right) ,t_{1}\right\} $ is a chain in $%
\mathbb{P}$ with least element $t_{\min }$, one has that $s\in S_{t}\cap
\varphi _{S}\left( S_{0}\right) $, the color of $\tau _{0}\left( s\right) $
is equal to the color of $\xi \left( f_{\tau _{0}}\left( t\right) \right) $
and the color of $\tau _{0}\left( s\right) +\xi _{1}\left( t_{1}\right) $ is
equal to the color of $\xi \left( t_{\min }\right) $. This allows one to
choose $x_{0}\left( t\right) \in S_{t_{0}}\cap \varphi _{S}\left(
S_{0}\right) $ satisfying ($1_{0}$) and ($2_{0}$).

We now consider the case $m=1$. In this case $\psi _{0}^{(\mathcal{F})}$
selects a finite subset $\mathcal{F}_{1}$ of $\mathcal{F}_{\alpha }$, and $%
\psi _{1}^{(S\mathcal{)}}$ selects a fintie subset $S_{1}=\psi
_{1}^{(S)}\left( x_{0}\right) $ that contains $\tau _{0}\left( x_{0}\left(
t\right) \right) $ for every $\tau _{0}\in \mathcal{F}_{0}$ and $t\in 
\mathbb{P}$. From ($2_{0}$) we deduce that

\begin{enumerate}
\item[($3_{0}$)] for every $t_{0},t_{1},t_{2}\in \mathbb{P}$ and $\tau
_{0}\in \mathcal{F}_{0}$, if $\left\{ f_{\tau _{0}}\left( t_{0}\right)
,t_{1},t_{2}\right\} $ is a chain in $\mathbb{P}$ with least element $t$,
then the color of $\tau _{0}\left( x_{0}\left( t_{0}\right) \right) +\xi
\left( t_{1}\right) +\xi \left( t_{2}\right) $ is equal to the color of $\xi
\left( t\right) $.
\end{enumerate}

Now fix $t\in \mathbb{P}$. We have that $\xi \left( t\right) s$, $\forall
\tau _{0}\in \mathcal{F}_{0}$, $\forall \tau _{1}\in \mathcal{F}_{1}$, $%
\forall t_{0},t_{2}\in \mathbb{P}$, one has that $s\in S_{t}\cap \varphi
_{S}\left( S_{1}\right) $, the color of $\tau _{1}\left( s\right) $ is equal
to the color of $\xi \left( f_{\tau }\left( t\right) \right) $, if $\left\{
f_{\tau _{0}}\left( t_{0}\right) ,f_{\tau _{1}}\left( t\right) \right\} $ is
a chain in $\mathbb{P}$ with least element $t_{\min }$ then the color of $%
\tau _{0}\left( x\left( t_{0}\right) \right) +\tau _{1}\left( s\right) $ is
equal to the color of $\xi \left( t_{\min }\right) $, if $\left\{ f_{\tau
_{0}}\left( t_{0}\right) ,f_{\tau _{1}}\left( t\right) ,t_{2}\right\} $ is a
chain in $\mathbb{P}$ with least element $t_{\min }$ then the color of $\tau
_{0}\left( x\left( t_{0}\right) \right) +\tau _{1}\left( s\right) +\xi
\left( t_{2}\right) $ is equal to the color of $\xi \left( t_{\min }\right) $%
, and if $\left\{ \tau _{1}\left( t\right) ,t_{2}\right\} $ is a chain in $%
\mathbb{P}$ with least element $t_{\min }$ then the color of $\tau
_{1}\left( s\right) +\xi \left( t_{2}\right) $ is equal to the color of $\xi
\left( t_{\min }\right) $. This allows one to choose $x_{1}\left( t\right)
\in S_{t}\cap \varphi _{S}\left( S_{1}\right) $ in such a way that ($1_{1}$)
and ($2_{1}$) are satisfied.

Suppose that a sequence as above has been defined up to $m$ in such a way
that ($1_{m}$) and ($2_{m}$) are satisfied. From ($2_{m}$) and the fact that 
$\xi $ is an order-preserving idempotent in $\left( \gamma S\right) _{\alpha
}$, it follows that the following holds as well:

\begin{enumerate}
\item[($3_{m}$)] for every $\ell \leq m$, $n_{0}<n_{1}<\cdots <n_{\ell }\leq
m$, $t_{i}\in \mathbb{P}$ for $i\leq \ell +2$, $\tau _{i}\in \psi _{n_{i}}^{(%
\mathcal{F})}\left( x_{0},\ldots ,x_{n_{i}-1}\right) $ for $i\leq \ell $
such that $\left\{ f_{\tau _{i}}\left( t_{i}\right) :i\leq \ell \right\}
\cup \left\{ t_{i+1},t_{i+2}\right\} $ is a chain in $\mathbb{P}$ with least
element $t$, one has that the color of $\tau _{0}\left( x_{n_{0}}\left(
t_{0}\right) \right) +\cdots +\tau _{\ell }\left( x_{n_{\ell }}\left(
t_{\ell }\right) \right) +\xi \left( t_{\ell +1}\right) +\xi \left( t_{\ell
+2}\right) $ is equal to the color of $\xi \left( t\right) $.
\end{enumerate}

Fix $t\in \mathbb{P}$. Using ($2_{m}$), ($3_{m}$) we see that $\xi \left(
t\right) s$, for every $\ell \leq m$, $n_{0}<n_{1}<\cdots <n_{\ell }\leq m$, 
$t_{i}\in \mathbb{P}$ for $i\leq \ell +2$, $\tau _{i}\in \psi _{n_{i}}^{(%
\mathcal{F})}\left( x_{0},\ldots ,x_{n_{i}-1}\right) $ for $i\leq \ell $ and 
$\tau \in \psi _{m+1}^{(\mathcal{F})}\left( x_{0},\ldots ,x_{m}\right) $, if 
$\left\{ f_{\tau _{i}}\left( t_{i}\right) :i\leq \ell \right\} \cup \left\{
f_{\tau }\left( t\right) \right\} $ is a chain in $\mathbb{P}$ with least
element $t_{\min }$ then the color of $\tau _{0}\left( x_{n_{0}}\left(
t_{0}\right) \right) +\cdots +\tau _{\ell }\left( x_{n_{\ell }}\left(
t_{\ell }\right) \right) +\tau \left( s\right) $ is equal to the color of $%
\xi (t_{\min })$, and if $\left\{ f_{\tau _{i}}\left( t_{i}\right) :i\leq
\ell \right\} \cup \left\{ f_{\tau }\left( t\right) ,t_{\ell +2}\right\} $
is a chain in $\mathbb{P}$ with least element $t_{\min }$ then the color of $%
\tau _{0}\left( x_{n_{0}}\left( t_{0}\right) \right) +\cdots +\tau _{\ell
}\left( x_{n_{\ell }}\left( t_{\ell }\right) \right) +\tau \left( s\right)
+\xi \left( t_{\ell +2}\right) $ is equal to the color of $\xi (t_{\min })$.
This allows one to choose $x_{m+1}\left( t\right) $ for every $t\in \mathbb{P%
}$ in such a way that ($1_{m+1}$) and ($2_{m+1}$) are satisfied. This
concludes the recursive construction.
\end{proof}

\subsection{Actions of trees on partial semigroups\label%
{Subsection:action-tree}}

Suppose that $T$ is a finite rooted tree. As in Subsection \ref%
{Subsection:action-compact-semigroup}, we consider $T$ as an ordered set
with respect to its canonical ordering. This is defined by setting $%
t_{0}\leq t_{1}$ if and only if $t_{0},t_{1}\in T$ and $t_{0}$ is a
descendent of $t_{1}$.

\begin{definition}
\label{Definition:Ramsey-action}Suppose that $\alpha $ is an action of a
finite rooted tree $T$ on an adequate partial semigroup $S$ as in Definition %
\ref{Definition:action-partial}. We say that $\alpha $ is Ramsey if, for
every $\tau \in \mathcal{F}_{\alpha }$, the corresponding spine $f_{\tau
}:T\rightarrow T$ is a regressive homomorphism, and for any finite subset $%
S_{0}$ of $S$, for any finite coloring $c$ of $S$, and any finite subset $%
\mathcal{F}_{0}$ of $\mathcal{F}_{\alpha }$, there exists a function $%
x:T\rightarrow S$ such that, for any $\tau \in \mathcal{F}_{0}$ and $t\in T$%
, $x\left( t\right) \in S_{t}\cap \varphi _{S}\left( S_{0}\right) $ and the
color of $\tau \left( x\left( t\right) \right) $ depends only on $f_{\tau
}\left( t\right) $.
\end{definition}

While the definition of adequate action might seem difficult to verify, it
holds trivially in many examples, including the case of the action
corresponding to Gowers' theorem for multiple tetris operations.

\begin{theorem}
\label{Theorem:monochromatic-tree-action}Suppose that $\alpha $ is an action
of a finite rooted tree on an adequate partial semigroup $S$ given by some
semigroup $\mathcal{F}_{\alpha }\subset \mathrm{End}\left( S\right) $ such
that, for every $\tau \in \mathcal{F}_{\alpha }$, the corresponding spine $%
f_{\tau }$ is a regressive homomorphism. The following statements are
equivalent:

\begin{enumerate}
\item $\alpha $ is Ramsey;

\item the action induced by $\alpha $ on $\gamma S$ is Ramsey;

\item for any finite coloring $c$ of $S$, sequence $(\psi _{n}^{(\mathcal{F}%
)})$ of functions $\psi _{n}^{(\mathcal{F})}:(S^{T})^{n}\rightarrow \left[ 
\mathcal{F}_{\alpha }\right] ^{<\aleph _{0}}$ and sequence $(\psi
_{n}^{(S)}) $ of functions $\psi _{n}^{(S)}:(S^{T})^{n}\rightarrow \left[ S%
\right] ^{<\aleph _{0}}$ such that $\psi _{n}\left( x_{0},\ldots
,x_{n-1}\right) $ contains the range of $\tau _{i}\circ x_{i}$ for every $%
i\in n$ and $\tau _{i}\in \psi _{i}^{(\mathcal{F})}\left( x_{0},\ldots
,x_{i-1}\right) $, there exist functions $x_{n}:T\rightarrow S$ such that

\begin{itemize}
\item $x_{n}\left( t\right) \in S_{t}\cap (\varphi _{S}\circ \psi
_{n}^{(S)})\left( x_{0},\ldots ,x_{n-1}\right) $ for every $n\in \omega $
and $t\in T$; and

\item for any $\ell \in \omega $, $n_{0}<n_{1}<\cdots <n_{\ell }\in \omega $%
, $t_{i}\in T$ for $i\leq \ell $, and $\tau _{i}\in \psi _{n}^{(\mathcal{F}%
)}\left( x_{0},\ldots ,x_{n_{i}-1}\right) $ for $i\leq \ell $, if $\left\{
f_{\tau _{i}}\left( t_{i}\right) :i\leq \ell \right\} $ is a chain in $T$
with least element $t$, then the color of $\tau _{0}\left( x_{n_{0}}\left(
t_{0}\right) \right) +\cdots +\tau _{\ell }\left( x_{n_{\ell }}\left(
t_{\ell }\right) \right) $ depends only on $t$.
\end{itemize}
\end{enumerate}
\end{theorem}

\begin{proof}
(1)$\Rightarrow $(2) Since $\alpha $ is Ramsey, we have that for any finite
subset $S_{0}$ of $S$, for any finite coloring $c$ of $S$, and any finite
subset $\mathcal{F}_{0}$ of $\mathcal{F}_{\alpha }$, there exists a function 
$x:T\rightarrow \varphi _{S}\left( S_{0}\right) $ such that, for any $\tau
\in \mathcal{F}_{0}$ and $t\in T$, the color of $\tau \left( x\left(
t\right) \right) $ depends only on $f_{\tau }\left( t\right) $. By
compactness of $\beta S$ we deduce that there exists a function $\xi
:T\rightarrow \gamma S$ such that, for any $\tau \in \mathcal{F}$, $t\in T$,
and any finite coloring $c$ of $S$, the color of $\tau \left( \xi \left(
t\right) \right) $ depends only on $f_{\tau }\left( t\right) $. This being
true for any coloring of $S$ implies that $\tau \left( \xi \left( t\right)
\right) =\xi \left( f_{\tau }\left( t\right) \right) $ for every $t\in T$
and $\tau \in \mathcal{F}_{\alpha }$. Therefore $\xi \in \left( \gamma
S\right) _{\alpha }$.

(2)$\Rightarrow $(1) Suppose that $\xi \in \left( \gamma S\right) _{\alpha }$%
. Fix finite subsets $S_{0}$ of $S$, $\mathcal{F}_{0}$ of $\mathcal{F}%
_{\alpha }$, and a finite coloring $c$ of $S$. Consider the canonical
extension of $c$ to a finite coloring of $\beta S$. We have that, for every $%
t\in T$ and $\tau \in \mathcal{F}_{0}$, $\tau \left( \xi \left( t\right)
\right) =\xi \left( f_{\tau }\left( t\right) \right) $. In particular the
color of $\tau \left( \xi \left( t\right) \right) $ is equal to the color of 
$\xi \left( f_{\tau }\left( t\right) \right) $. Fix $t\in T$. Using the
notation of ultrafilter quantifiers, we have that $\xi \left( t\right) s$, $%
\forall \tau \in \mathcal{F}_{0}$, $s\in S_{t}\cap \varphi _{S}\left(
S_{0}\right) $ and the color of $\tau \left( s\right) $ is equal to the
color of $\xi \left( f_{\tau }\left( t\right) \right) $. Therefore we can
choose an element $x\left( t\right) \in S_{t}\cap \varphi _{S}\left(
S_{0}\right) $ for every $t\in T$ such that the function $x:T\rightarrow S$
witnesses that the action $\alpha $ is Ramsey.

(3)$\Rightarrow $(1) Observe that the definition of Ramsey action is the
particular instance of (4) where the sequence $\left( x_{n}\right) $ has
length $1$.

(3)$\Rightarrow $(4) This is a consequence of Proposition \ref%
{Proposition:tree} and Theorem \ref{Theorem:monochromatic-action}.
\end{proof}

In the Section \ref{Section:examples} we will explain how various results in
the literature can be seen as a special instance of Theorem \ref%
{Theorem:monochromatic-tree-action}.

\subsection{Products of actions\label{Subsection:products}}

We recall the notion of tensor product \cite[Section 11.1]%
{hindman_algebra_2012}---see also \cite[Section 1.2]%
{todorcevic_introduction_2010}---of ultrafilters. Suppose that $\mathcal{U}$
and $\mathcal{V}$ are ultrafilter on sets $A,B$, respectively. Then $%
\mathcal{U}\otimes \mathcal{V}$ is the ultrafilter on $A\times B$ obtained
by declaring, for any $C\subset A\times B$, $C\in \mathcal{U}\otimes 
\mathcal{V}$ if and only if $\left( \mathcal{U}a\right) $, $\left( \mathcal{V%
}b\right) $, $\left( a,b\right) \in C$. Observe that in the particular case
when $\mathcal{U}$ is the principal ultrafilter over $a\in A$, then $C\in 
\mathcal{U}\otimes \mathcal{V}$ if and only if $\left\{ b\in B:\left(
a,b\right) \in C\right\} \in \mathcal{V}$. In the following we identify an
element $a$ of a set $A$ with the corresponding principal ultrafilter. The
following result is proved in \cite[Corollary 2.8]{bergelson_polynomial_2014}

\begin{theorem}[Bergelson--Hindman--Williams]
\label{Theorem:BH-tensor}Suppose that $m,k\in \mathbb{N}$ and $\lambda
:m\rightarrow k$ is a function. Suppose that for $i\in k$, $\left(
S_{i},+_{i}\right) $ is a semigroup and $p_{i}\in \gamma S_{i}$ is an
idempotent element. Set $S:=S_{\lambda \left( 0\right) }\times \cdots \times
S_{\lambda \left( m-1\right) }$ and suppose that $A$ is a subset of $S$. If $%
A\in p_{\lambda \left( 0\right) }\otimes \cdots \otimes p_{\lambda \left(
m-1\right) }$ then for each $i\in k$ there exist a sequence $\left(
y_{i,n}\right) _{n\in \omega }$ in $S_{i}$ such that%
\begin{equation*}
\left\{ \left( \sum\nolimits_{d\in F_{s}}^{\lambda \left( s\right)
}y_{\lambda \left( s\right) ,d}\right) _{s\in m}:F_{0}<F_{1}<\cdots <F_{m-1}%
\text{ are finite subsets of }\omega \right\}
\end{equation*}%
is contained in $A$.
\end{theorem}

In the statement of Theorem \ref{Theorem:BH-tensor} and in the following,
the expression $\sum_{d\in F}^{j}y_{d}$, where $\left( y_{n}\right) $ is a
sequence in a partial semigroup $\left( S,+_{j}\right) $, denotes the
element $y_{d_{0}}+_{j}y_{d_{1}}+_{j}\cdots +_{j}y_{d_{\ell }}$ of $S$,
where $\left\{ d_{0},\ldots ,d_{\ell }\right\} $ is an increasing
enumeration of $F$. Whenever we write such an expression, we will also
implicitly assert that $y_{d_{0}}+_{j}y_{d_{1}}+_{j}\cdots +_{j}y_{d_{\ell
}} $ is defined in $\left( S,+_{j}\right) $.

We will now present a generalization of Theorem \ref{Theorem:BH-tensor} to
the setting of actions of ordered sets on adequate partial semigroups.
Suppose that $m,k\in \mathbb{N}$ and $\lambda :m\rightarrow k$ is a
function. For each $i\in k$ let $\left( S_{i},+_{i}\right) $ be an\emph{\ }%
adequate partial semigroup, and $\mathbb{P}_{i}$ be a \emph{finite} ordered
set. For every $i\in k$ we can consider an action $\alpha _{i}$ of $\mathbb{P%
}_{i}$ on $\left( S_{i},+_{i}\right) $. This is given by a function $\mathbb{%
P}_{i}\rightarrow \mathcal{S}\left( S_{i},+_{i}\right) $, $t\mapsto S_{i,t}$
such that $S_{i,t}\leq S_{i,t^{\prime }}$ whenever $t\leq t^{\prime }$, and
a semigroup $\mathcal{F}_{\alpha _{i}}$ of adequate partial semigroup
homomorphisms of $\left( S_{i},+_{i}\right) $ such that for every $\tau \in 
\mathcal{F}_{\alpha _{i}}$ there exists a function $f_{\tau }:\mathbb{P}%
_{i}\rightarrow \mathbb{P}_{i}$ with the property that $\tau $ maps \ $%
S_{i,t}$ to $S_{i,f_{\tau }\left( t\right) }$ for every $t\in \mathbb{P}_{i}$%
. Recall that each action $\alpha _{i}$ extends to an action of $\mathbb{P}%
_{i}$ on $\gamma S_{i}$. In this case $\left( \gamma S_{i}\right) _{\alpha
_{i}}$ denotes the set of functions $\xi :\mathbb{P}_{i}\rightarrow \gamma
S_{i}$ such that, for any $t,t^{\prime }\in \mathbb{P}$, $\xi \left(
t\right) $ is an idempotent element of $\gamma S_{i,t}$ and $\xi \left(
t\right) \leq \xi \left( t^{\prime }\right) $ if $t\leq t^{\prime }$. The
following result is the natural generalization of Theorem \ref%
{Theorem:BH-tensor} to the setting of actions of ordered sets on partial
semigroups.

\begin{theorem}
\label{Theorem:product-action}Suppose that $\xi _{i}\in \left( \gamma
S_{i}\right) _{\alpha _{i}}$ is an order-preserving idempotent for $i\in k$.
Set $S:=S_{\lambda \left( 0\right) }\times \cdots \times S_{\lambda \left(
m-1\right) }$ and suppose that $c$ is a finite coloring of $S$. Extend $c$
canonically to a coloring of $\beta S$. Fix for $i\in k$ a sequence $(\psi
_{i,n}^{(\mathcal{F})})$ of functions $\psi _{i,n}^{(\mathcal{F})}:(S_{i}^{%
\mathbb{P}_{i}})^{n}\rightarrow \left[ \mathcal{F}_{\alpha _{i}}\right]
^{<\aleph _{0}}$ and a sequence $(\psi _{i,n}^{(S)})$ of functions $\psi
_{i,n}^{(S)}:(S_{i}^{\mathbb{P}_{i}})^{n}\rightarrow \left[ S_{i}\right]
^{<\aleph _{0}}$ such that $\psi _{i,n}\left( x_{i,0},\ldots
,x_{i,n-1}\right) $ contains the range of $\tau _{j}\circ x_{j}$ for every $%
j\in n$ and $\tau _{j}\in \psi _{i,j}^{(\mathcal{F})}\left( x_{i,0},\ldots
,x_{i,j-1}\right) $. Then there exist sequences $\left( x_{i,n}\right)
_{n\in \omega }$ of functions $x_{i,n}:\mathbb{P}_{i}\rightarrow S_{i}$ for $%
i\in k$ such that

\begin{itemize}
\item $x_{i,n}\left( t\right) \in S_{i,t}\cap (\varphi _{S_{i}}\circ \psi
_{i,n}^{(S)})\left( x_{i,0},\ldots ,x_{i,n-1}\right) $ for every $n\in
\omega $, $i\in k$, and $t\in \mathbb{P}_{i}$; and

\item for any $\ell \in \omega $, $i\in k$, finite subsets $%
F_{0}<F_{1}<\cdots <F_{\ell -1}$ of $\omega $, sequences $\left(
t_{i,n}\right) _{n\in \omega }$ in $\mathbb{P}_{i}$ and $\left( \tau
_{i,n}\right) _{n\in \omega }$ in $\mathcal{F}_{\alpha _{i}}$ such that $%
\tau _{i,n}\in \psi _{i,n}^{(\mathcal{F})}\left( x_{i,0},\ldots
,x_{i,n-1}\right) $ for every $n\in \omega $, if $\{f_{\tau _{\lambda \left(
s\right) ,d}}\left( t_{\lambda \left( s\right) ,d}\right) :d\in F_{s}\}$ is
a chain in $\mathbb{P}_{\lambda \left( s\right) }$ with least element $t_{s}$
for every $s\in m$, then the color of%
\begin{equation*}
\left( \sum\nolimits_{d\in F_{s}}^{\lambda \left( s\right) }\tau _{\lambda
(s),d}\left( x_{\lambda (s),d}\left( t_{\lambda \left( s\right) ,d}\right)
\right) \right) _{s\in m}
\end{equation*}%
is the same as the color of $\xi _{\lambda (0)}\left( t_{0}\right) \otimes
\xi _{\lambda \left( 1\right) }\left( t_{1}\right) \otimes \cdots \otimes
\xi _{\lambda \left( m-1\right) }\left( t_{m-1}\right) $.
\end{itemize}
\end{theorem}

\begin{proof}
We define by recursion a $\ell \in \omega $ functions $x_{i,\ell }:\mathbb{P}%
_{i}\rightarrow S_{i}$ with $x_{i,\ell }\left( t\right) \in S_{i,t}\cap
(\varphi _{S_{i}}\circ \psi _{i,\ell }^{(S)})\left( x_{i,0},\ldots
,x_{i,\ell -1}\right) $ for $i\in k$ such that for every $\ell \in \omega $
the following holds:

\begin{enumerate}
\item[($1_{\ell }$)] for every $j\in m$, finite subsets $F_{0}<F_{1}<\cdots
<F_{j}<\ell $ of $\omega $, sequences $\left( t_{i,d}\right) _{d\in \ell }$
in $\mathbb{P}_{i}$ and $\left( \tau _{i,d}\right) _{d\in \ell }$ in $%
\mathcal{F}_{\alpha _{i}}$ such that $\tau _{i,d}\in \psi _{i,d}^{(\mathcal{F%
})}\left( x_{0},\ldots ,x_{d-1}\right) $ for every $d\in \ell $, and $%
h_{s}\in \mathbb{P}_{\lambda \left( s\right) }$ for $j<s<m$, if $\{f_{\tau
_{\lambda \left( s\right) ,d}}\left( t_{\lambda \left( s\right) ,d}\right)
:d\in F_{s}\}$ is a chain in $\mathbb{P}_{\lambda \left( s\right) }$ with
least element $t_{s}$ for every $s\leq j$, then the color of%
\begin{equation*}
\left( \sum\nolimits_{d\in F_{0}}^{\lambda \left( 0\right) }\tau _{\lambda
(0),d}\left( x_{\lambda (0),d}\left( t_{0}\right) \right) \right) \otimes
\cdots \otimes \left( \sum\nolimits_{d\in F_{j}}^{\lambda \left( j\right)
}\tau _{\lambda (j),d}\left( x_{\lambda (j),d}\left( t_{j}\right) \right)
\right) \otimes \xi _{\lambda (j+1)}\left( h_{j+1}\right) \otimes \cdots
\otimes \xi _{\lambda \left( m-1\right) }\left( h_{m-1}\right)
\end{equation*}%
(where we identify an element of $S_{i}$ with the corresponding principal
ultrafilter) is the same as the color of 
\begin{equation*}
\xi _{\lambda (0)}\left( t_{0}\right) \otimes \xi _{\lambda \left( 1\right)
}\left( t_{1}\right) \otimes \cdots \otimes \xi _{\lambda \left( j\right)
}\left( t_{j}\right) \otimes \xi _{\lambda \left( j+1\right) }\left(
h_{j+1}\right) \otimes \cdots \otimes \xi _{\lambda \left( m-1\right)
}\left( h_{m-1}\right) \text{,}
\end{equation*}%
and

\item[($2_{\ell }$)] for every $j\in m$, finite subsets $F_{0}<F_{1}<\cdots
<F_{j}<\ell $, sequences $\left( t_{i,d}\right) _{d\leq \ell }$ in $\mathbb{P%
}_{i}$ and $\left( \tau _{i,d}\right) _{d\in \ell }$ in $\mathcal{F}_{\alpha
_{i}}$ such that $\tau _{i,d}\in \psi _{i,d}^{(\mathcal{F})}\left(
x_{0},\ldots ,x_{d-1}\right) $ for every $d\in \ell $, and $h_{s}\in \mathbb{%
P}_{\lambda \left( s\right) }$ for $j<s<m$, if $\{f_{\tau _{\lambda \left(
s\right) ,d}}\left( t_{\lambda \left( s\right) ,d}\right) :d\in F_{s}\}$ is
a chain in $\mathbb{P}_{\lambda \left( s\right) }$ with least element $t_{s}$
for every $s\in j$ and $\{f_{\tau _{\lambda \left( j\right) ,d}}\left(
t_{\lambda \left( j\right) ,d}\right) :d\in F_{j}\}\cup \left\{ t_{\lambda
\left( j\right) ,\ell }\right\} $ is a chain in $\mathbb{P}_{\lambda \left(
j\right) }$ with least element $t_{j}$, then the color of%
\begin{eqnarray*}
&&\left( \sum\nolimits_{d\in F_{0}}^{\lambda \left( 0\right) }\tau _{\lambda
(0),d}\left( x_{\lambda (0),d}\left( t_{\lambda (0),d}\right) \right)
\right) \otimes \cdots \otimes \left( \sum\nolimits_{d\in F_{j-1}}^{\lambda
\left( j-1\right) }\tau _{\lambda (j-1),d}\left( x_{\lambda (j-1),d}\left(
t_{\lambda (j-1),d}\right) \right) \right) \otimes \\
&&\otimes \left( \sum\nolimits_{d\in F_{j}}^{\lambda \left( j\right) }\tau
_{\lambda (j),d}\left( x_{\lambda (j),d}\left( t_{\lambda (j),d}\right)
\right) +_{\lambda \left( j\right) }\xi _{\lambda \left( j\right)
}(t_{\lambda \left( j\right) ,\ell })\right) \otimes \xi _{\lambda
(j+1)}\left( h_{j+1}\right) \otimes \cdots \otimes \xi _{\lambda \left(
m-1\right) }\left( h_{m-1}\right)
\end{eqnarray*}%
is the same as the color of 
\begin{equation*}
\xi _{\lambda (0)}\left( t_{0}\right) \otimes \xi _{\lambda \left( 1\right)
}\left( t_{1}\right) \otimes \cdots \otimes \xi _{\lambda \left( j\right)
}\left( t_{j}\right) \otimes \xi _{\lambda \left( j+1\right) }\left(
h_{j+1}\right) \otimes \cdots \otimes \xi _{\lambda \left( m-1\right)
}\left( h_{m-1}\right) \text{,}
\end{equation*}
\end{enumerate}

Suppose that such a sequence has been defined up to $\ell $ in such a way
that ($1_{\ell }$) and ($2_{\ell }$) are satisfied. From ($2_{\ell }$) and
the fact that $\xi _{i}$ for $i\in k$ is an order-preserving idempotent in $%
\left( \gamma S_{i}\right) _{\alpha _{i}}$, it follows that the following
holds as well:

\begin{enumerate}
\item[($3_{\ell }$)] for every $j\in m$, finite subsets $F_{0}<F_{1}<\cdots
<F_{j}<\ell $, sequences $\left( t_{i,d}\right) _{d\leq \ell +1}$ in $%
\mathbb{P}_{i}$ and $\left( \tau _{i,d}\right) _{d\in \ell }$ in $\mathcal{F}%
_{\alpha _{i}}$ such that $\tau _{i,d}\in \psi _{i,d}^{(\mathcal{F})}\left(
x_{0},\ldots ,x_{d-1}\right) $ for every $d\leq \ell $, and $h_{s}\in 
\mathbb{P}_{\lambda \left( s\right) }$ for $j<i<m$, if $\{f_{\tau _{\lambda
\left( s\right) ,d}}\left( t_{\lambda \left( s\right) ,d}\right) :d\in
F_{s}\}$ is a chain in $\mathbb{P}_{\lambda \left( s\right) }$ with least
element $t_{s}$ for every $s\in j$, and $\{f_{\tau _{\lambda \left( j\right)
,d}}\left( t_{\lambda \left( j\right) ,d}\right) :d\in F_{j}\}\cup \left\{
t_{\lambda \left( j\right) ,\ell },t_{\lambda \left( j\right) ,\ell
+1}\right\} $ is a chain in $\mathbb{P}_{\lambda \left( j\right) }$ with
least element $t_{j}$, then the color of%
\begin{eqnarray*}
&&\left( \sum\nolimits_{d\in F_{0}}^{\lambda \left( 0\right) }\tau _{\lambda
(0),d}\left( x_{\lambda (0),d}\left( t_{\lambda (0),d}\right) \right)
\right) \otimes \cdots \otimes \sum\nolimits_{d\in F_{j-1}}^{\lambda \left(
j-1\right) }\tau _{\lambda (j-1),d}\left( x_{\lambda (j-1),d}\left(
t_{\lambda (j-1),d}\right) \right) \otimes \\
&&\otimes \left( \sum\nolimits_{d\in F_{j}}^{\lambda \left( j\right) }\tau
_{\lambda (j),d}\left( x_{\lambda (j),d}\left( t_{\lambda (j),d}\right)
\right) +_{\lambda \left( j\right) }\xi _{\lambda \left( j\right)
}(t_{\lambda \left( j\right) ,\ell })+_{\lambda \left( j\right) }\xi
_{\lambda \left( j\right) }\left( t_{\lambda \left( j\right) ,\ell
+1}\right) \right) \otimes \xi _{\lambda (j+1)}\left( h_{j+1}\right) \otimes
\cdots \otimes \xi _{\lambda \left( m-1\right) }\left( h_{m-1}\right)
\end{eqnarray*}%
is the same as the color of 
\begin{equation*}
\xi _{\lambda (0)}\left( t_{0}\right) \otimes \xi _{\lambda \left( 1\right)
}\left( t_{1}\right) \otimes \cdots \otimes \xi _{\lambda \left( j\right)
}\left( t_{j}\right) \otimes \xi _{\lambda \left( j+1\right) }\left(
h_{j+1}\right) \otimes \cdots \otimes \xi _{\lambda \left( m-1\right)
}\left( h_{m-1}\right) \text{,}
\end{equation*}
\end{enumerate}

Considering the definition of the operations between ultrafilters, one can
then see that, for every $i\in k$ and $t\in \mathbb{P}_{i}$,

\begin{itemize}
\item if follows from ($2_{m}$) that the set of possible choices of $%
x_{i,\ell +1}\left( t\right) \in S_{i,t}$ satisfying ($1_{\ell +1}$)
whenever $s\in m$, $\lambda \left( s\right) =i$, and $t_{\lambda \left(
s\right) ,\ell }=t$, belongs to $\xi _{i}\left( t\right) $, and

\item it follows from ($3_{m}$) that the set of possible choices of $%
x_{i,\ell +1}\left( t\right) \in S_{i,t}$ satisfying ($2_{\ell +1}$)
whenever $s\in m$, $\lambda \left( s\right) =i$, and $t_{\lambda \left(
s\right) ,\ell }=t$ belongs to $\xi _{i}\left( t\right) $.
\end{itemize}

Therefore one can choose $x_{i,m+1}:\mathbb{P}\rightarrow S_{i}$ with $%
x_{i,m+1}\left( t\right) \in S_{i,t}\cap (\varphi _{S_{i}}\circ \psi
_{i,m+1}^{(S)})\left( x_{i,0},\ldots ,x_{i,m}\right) $ for $i\in k$ such
that both ($1_{m+1}$) and ($2_{m+1}$) are satisfied for every $t\in \mathbb{P%
}$. This concludes the recursive construction.
\end{proof}

\begin{corollary}
\label{Corollary:product-action-tree}Suppose that $T_{i}$ for $i\in k$ are
finite rooted trees, and $\alpha _{i}$ for $i\in k$ is a Ramsey action of $%
T_{i}$ on the adequate partial semigroup $\left( S_{i},+_{i}\right) $. Set $%
S:=S_{\lambda \left( 0\right) }\times \cdots \times S_{\lambda \left(
m-1\right) }$ and suppose that $c$ is a finite coloring of $S$. Fix for $%
i\in k$ a sequence $(\psi _{i,n}^{(\mathcal{F})})$ of functions $\psi
_{i,n}^{(\mathcal{F})}:(S_{i}^{T_{i}})^{n}\rightarrow \left[ \mathcal{F}%
_{\alpha _{i}}\right] ^{<\aleph _{0}}$ and a sequence $(\psi _{i,n}^{(S)})$
of functions $\psi _{i,n}^{(S)}:(S_{i}^{T_{i}})^{n}\rightarrow \left[ S_{i}%
\right] ^{<\aleph _{0}}$ such that $\psi _{i,n}\left( x_{i,0},\ldots
,x_{i,n-1}\right) $ contains the range of $\tau _{j}\circ x_{j}$ for every $%
j\in n$ and $\tau _{j}\in \psi _{i,j}^{(\mathcal{F})}\left( x_{i,0},\ldots
,x_{i,j-1}\right) $. Then there exist sequences $\left( x_{i,n}\right)
_{n\in \omega }$ of functions $x_{i,n}:T_{i}\rightarrow S_{i}$ such that

\begin{itemize}
\item $x_{i,n}\left( t\right) \in S_{i,t}\cap (\varphi _{S_{i}}\circ \psi
_{i,n}^{(S)})\left( x_{i,0},\ldots ,x_{i,n-1}\right) $ for every $n\in
\omega $, $i\in k$, and $t\in T_{i}$; and

\item for any $i\in k$, finite subsets $F_{0}<F_{1}<\cdots <F_{m-1}$ of $%
\omega $, sequences $\left( t_{i,n}\right) _{n\in \omega }$ in $T_{i}$ and $%
\left( \tau _{i,n}\right) _{n\in \omega }$ in $\mathcal{F}_{\alpha _{i}}$
such that $\tau _{i,n}\in \psi _{i,n}^{(\mathcal{F})}\left( x_{0},\ldots
,x_{n-1}\right) $ for every $n\in \omega $, if $\{f_{\tau _{\lambda \left(
s\right) ,d}}\left( t_{i,d}\right) :d\in F_{\lambda \left( s\right) }\}$ is
a chain in $T_{\lambda \left( s\right) }$ with least element $t_{s}$ for
every $s\in m$, then the color of%
\begin{equation*}
\left( \sum\nolimits_{d\in F_{i}}^{\lambda \left( s\right) }\tau _{\lambda
(i),d}\left( x_{\lambda (i),d}\left( t_{\lambda (i),d}\right) \right)
\right) _{s\in m}
\end{equation*}%
depends only on $\left( t_{s}\right) _{s\in m}$.
\end{itemize}
\end{corollary}

\section{Examples of actions of trees on partial semigroups\label%
{Section:examples}}

\subsection{Gowers' theorem for multiple tetris operations\label%
{Subsection:G}}

Suppose that $T$ is a finite rooted tree with root $r$, and let $T^{+}$ be
the set of nodes of $T$ different from the root. We regard as above $T$ as
an ordered set with respect to its canonical tree ordering. Let $\mathrm{%
\mathrm{FI}N}^{T}$ be the space of functions $b:\mathrm{dom}\left( b\right)
\rightarrow T^{+}$ where $\mathrm{dom}\left( b\right) $ is a finite
(possibly empty) subset of $\omega $ and the range of $b$ is contained in a
branch of $T$. Then $\mathrm{FIN}^{T}$ is a partial semigroup with respect
of the partially defined operation $\left( b_{0},b_{1}\right) \mapsto
b_{0}+b_{1}$, which is defined whenever the intersection of the domains of $%
b_{0},b_{1}$ is empty.\ In this case, $b_{0}+b_{1}$ is defined to be the
union of $b_{0}$ and $b_{1}$ (where we identify functions with their graph).

There is a natural action of $T$ on $\mathrm{\mathrm{FIN}}^{T}$ defined as
follows. For every $t\in T^{+}$ let $\mathrm{\mathrm{FIN}}_{t}^{T}$ be the
space of $b\in \mathrm{\mathrm{FIN}}^{T}$ such that the range of $b$ is a
chain in $T$ with least element $t$. We also let $\mathrm{FIN}_{r}^{T}$ be
the set containing only the empty function. Any regressive homomorphism $%
f:T\rightarrow T$ defines a canonical adequate partial semigroup
homomorphism $\tau _{f}:\mathrm{FIN}^{T}\rightarrow \mathrm{FIN}^{T}$ as
follows. Suppose that $b\in \mathrm{FIN}^{T}$. Then $\tau _{f}\left(
b\right) $ has domain $\left\{ n\in \mathrm{dom}\left( b\right) :f\left(
b\left( n\right) \right) \in T^{+}\right\} $ and it is defined by $\tau
_{f}\left( b\right) :n\mapsto f\left( b\left( n\right) \right) $. Observe
that $\tau _{f}$ maps $\mathrm{FIN}_{t}^{T}$ to $\mathrm{FIN}_{f\left(
t\right) }^{T}$ for every $t\in T$. Therefore $\tau _{f}$ has \emph{spine }$%
f $ according to Definition \ref{Definition:action-partial}. The collection $%
\mathcal{F}_{\alpha }$ of adequate partial semigroups homomorphisms $\tau
_{f}$ where $f$ varies among all the regressive homomorphisms $%
f:T\rightarrow T$ is an action of $T$ on the adequate partial semigroup $%
\mathrm{FIN}^{T}$ according to Definition \ref{Definition:action-partial}.
It is easily seen that such an action is Ramsey as in Definition \ref%
{Definition:Ramsey-action}. Indeed consider a finite subset $S_{0}$ of $%
\mathrm{\mathrm{FI}N}^{T}$ and a finite coloring $c$ of $\mathrm{FIN}^{T}$.
Fix any nonempty finite subset $A$ of $\omega $ disjoint from the union of
the supports of the elements of $S_{0}$. Consider then the function $%
x:T\rightarrow \mathrm{FIN}^{T}$, $t\mapsto b_{t}$ defined by setting $%
b_{t}\in \mathrm{FIN}_{t}^{T}$ to be the function with domain $A$ and
constantly equal to $t$. This witnesses that such an action is Ramsey. From
Theorem \ref{Theorem:monochromatic-tree-action} one can deduce the following.

\begin{theorem}
\label{Theorem:monochromatic-tree-actions}Suppose that $T$ is a finite
rooted tree. For any finite coloring $c$ of $\mathrm{FIN}^{T}$, for any
sequence $\left( d_{n}\right) $ of elements of $\omega $, there exists
sequences $\left( b_{t,n}\right) _{n\in \omega }$ for $t\in T$ of elements $%
b_{t,n}\in \mathrm{FIN}_{t}^{T}$ such that

\begin{itemize}
\item $\mathrm{dom}\left( b_{t,n}\right) >d_{n}$ and $\mathrm{dom}\left(
b_{t,n+1}\right) >\mathrm{dom}\left( b_{t,n}\right) $ for every $t\in T$ and 
$n\in \omega $;

\item for any $\ell \in \omega $, $n_{0}<n_{1}<\cdots <n_{\ell }\in \omega $%
, $t_{0},\ldots ,t_{\ell }\in T$, and regressive homomorphisms $%
f_{i}:T\rightarrow T$ for $i\leq \ell $, if $\left\{ f_{i}\left(
t_{i}\right) :i\leq \ell \right\} $ is a chain in $T$ with least element $t$%
, then the color of $\tau _{f_{0}}(b_{t_{0},n_{0}})+\cdots +\tau _{f_{\ell
}}(b_{t_{\ell },n_{\ell }})$ depends only on $t$.
\end{itemize}
\end{theorem}

From Corollary \ref{Corollary:product-action-tree} one can deduce the
following generalization of Theorem \ref{Theorem:monochromatic-tree-actions}.

\begin{theorem}
\label{Theorem:product-tree-action}Suppose that $m,k\in \mathbb{N}$ and $%
\lambda :m\rightarrow k$ is a function. Let $T_{i}$ for $i\in k$ be a finite
rooted tree. Fix any coloring $c$ of $\mathrm{FIN}^{T_{\lambda \left(
0\right) }}\times \cdots \times \mathrm{FIN}^{T_{\lambda \left( m-1\right)
}} $. For any sequence $\left( d_{n}\right) $ of elements of $\omega $,
there exist sequences $\left( b_{i,t,n}\right) _{n\in \omega }$ in $\mathrm{%
FIN}_{t}^{T_{i}}$ for $i\in k$ and $t\in T_{i}$ such that

\begin{itemize}
\item $\mathrm{dom}\left( b_{i,t,n}\right) >d_{n}$ and $\mathrm{dom}\left(
b_{i,t,n+1}\right) >\mathrm{dom}\left( b_{i,t,n}\right) $ for every $i\in k$%
, $t\in T_{i}$, and $n\in \omega $;

\item for finite subsets $F_{0}<F_{1}<\cdots <F_{m-1}$ of $\omega $, nodes $%
t_{i,d}$ of $T_{i}$ and regressive homomomrphisms $f_{i,d}$ of $T_{i}$ for $%
d\in \omega $ such that $\left\{ f_{\lambda \left( s\right) ,d}\left(
t_{\lambda \left( s\right) ,d}\right) :d\in F_{s}\right\} $ is a chain in $%
T_{\lambda \left( s\right) }$ with least element $t_{s}$ for every $s\in m$,
the color of 
\begin{equation*}
\left( \sum_{d\in F_{s}}\tau _{f_{\lambda \left( s\right) ,d}}(b_{f\left(
s\right) ,t_{\lambda \left( s\right) },d})\right) _{s\in m}
\end{equation*}%
depends only on $\left( t_{s}\right) _{s\in m}$.
\end{itemize}
\end{theorem}

Gowers' theorem for multiple tetris operations \cite[Theorem 1.1]%
{lupini_gowers_2016} is the particular instance of Theorem \ref%
{Theorem:monochromatic-tree-action} where $T$ is the rooted tree $I_{k}$
with $k+1$ nodes such that the canonical ordering on $I_{k}$ is a linear
order. In other words, $I_{k}$ is just the tree whose nodes are the initial
segments of $k$ ordered by reverse inclusion. The original version of
Gowers' theorem \cite{gowers_lipschitz_1992} corresponds to the case when
the only regressive homomomorphisms of $I_{k}$ considered are the ones
mapping a node its $j$-predecessor for some $j\geq 0$ (with the convention
that the $j$-th predecessor of a node of height at most $j$ is equal to the
root).

\subsection{The Hales-Jewett theorem for located words\label{Subsection:HJ}}

Let $L$ be a set (\emph{alphabet}) and $v$ a symbol not in $L$ (\emph{%
variable}). A \emph{located word }in the alphabet $L$ is a finitely
supported function $b:\mathrm{dom}\left( b\right) \rightarrow L$ where $%
\mathrm{dom}\left( b\right) $ is a (possibly empty) finite subset of $\omega 
$. Similarly, a \emph{located variable word }in the alphabet $L$ and
variable $v$ is a finitely supported function $b:\mathrm{dom}\left( b\right)
\rightarrow L\cup \left\{ v\right\} $ whose range contains $v$, where $%
\mathrm{dom}\left( b\right) $ is a finite subset of $\omega $. Following 
\cite[Section 2.6]{todorcevic_introduction_2010}, we let $\mathrm{FIN}_{L}$
be the set of located words in $L$, and $\mathrm{FIN}_{Lv}$ be the set of
located variable words in $L$ and the variable $v$. Then $S:=\mathrm{FIN}%
_{L}\cup \mathrm{FIN}_{Lv}$ has a natural partial semigroup operation,
obtained by letting $b_{0}+b_{1}$ be defined whenever the domains of $b_{0}$
and $b_{1}$ are disjoint. In such a case, $b_{0}+b_{1}$ is just $b_{0}\cup
b_{1}\ $(where we identify a function with its graph). It is clear that $%
\mathrm{FIN}_{L}$ is an adequate subsemigroup of $S$, while $\mathrm{FIN}%
_{Lv}$ is an adequate ideal of $S$.

Let $I_{1}$ be the rooted tree with only two nodes (including the root). We
denote the root of $I_{1}$ by $r$, and the other node by $t$. Every element $%
a$ of $L$ defines an adequate semigroup homomorphism $\tau _{a}:S\rightarrow
S$ obtained by replacing every occurrence of the variable $v$ with $a$. This
defines an action $\alpha $ of $I_{1}$ on $S$, where the closed subsemigroup
corresponding to the root $r$ of $I_{1}$ is $\mathrm{FIN}_{L}$, and the
closed subsemigroup corresponding to the other node $t$ of $I_{1}$ is $%
\mathrm{FIN}_{Lv}$. In order to see that such an action is Ramsey in the
sense of Definition \ref{Definition:Ramsey-action}, consider any minimal
idempotent $p_{r}$ of $\gamma \mathrm{FIN}_{L}$. Then by Lemma \ref%
{Lemma:minimal-lift} there exists a minimal idempotent $p_{t}$ of $\gamma 
\mathrm{FIN}_{Lv}$ such that $p_{t}\leq p_{r}$. By minimality of $p_{r}$ we
deduce that $\tau _{a}\left( p_{t}\right) =p_{r}$ for every $a\in L$.
Therefore the pair $\left( p_{r},p_{t}\right) $ witnesses that $\left(
\gamma S\right) _{\alpha }$ is nonempty by Theorem \ref%
{Theorem:monochromatic-tree-action}.

The infinite Hales-Jewett theorem for located words \cite[Theorem 6.1]%
{bergelson_partition_1994}---see also \cite[Theorem 2.40]%
{todorcevic_introduction_2010}---is then a consequence of Theorem \ref%
{Theorem:monochromatic-tree-action} applied to such a Ramsey action $\alpha $%
. From Corollary \ref{Corollary:product-action-tree} one can deduce the
following version.

\begin{theorem}
\label{Theorem:product-HJ}Suppose that $m,k\in \mathbb{N}$ and $\lambda
:m\rightarrow k$ is a function. Let $L_{i}$ for $i\in k$ be a set. Fix a
coloring of $\mathrm{FIN}_{L_{\lambda \left( 0\right) }}\times \cdots \times 
\mathrm{FIN}_{L_{\lambda (m-1)}}$. For any sequence $\left( d_{n}\right) $
of elements of $\omega $, and sequences $\left( L_{i,n}\right) _{n\in \omega
}$ of finite subsets of $L_{i}$ for $i\in k$, there exists sequences $\left(
b_{i,n}\right) _{n\in \omega }$ in $\mathrm{FIN}_{L_{i},v}$ for $i\in k$
such that

\begin{itemize}
\item $\mathrm{dom}\left( b_{i,n}\right) >d_{n}$ and $\mathrm{dom}\left(
b_{i,n+1}\right) >\mathrm{\mathrm{do}m}\left( b_{i,n}\right) $ for every $%
i\in k$ and $n\in \omega $;

\item the set of elements of $\mathrm{FIN}_{L_{\lambda \left( 0\right)
}}\times \cdots \times \mathrm{FIN}_{L_{\lambda (m-1)}}$ of the form%
\begin{equation*}
\left( \sum_{d\in F_{s}}\tau _{a_{f\left( s\right) ,d}}(b_{f\left( s\right)
,d})\right) _{s\in m}
\end{equation*}%
for $\ F_{0}<F_{1}<\cdots <F_{m-1}$ and $a_{i,d}\in L_{i}$ for $d\in \omega $
and $i\in k$ is monochromatic.
\end{itemize}
\end{theorem}

More generally, one can regard any\emph{\ layered action} on a partial
semigroup $S$ in the sense of \cite[Definition 3.3]{farah_partition_2002} as
an action of the tree $I_{k}$ on $S$ in the sense of Definition \ref%
{Definition:action-partial}. Any such an action is automatically Ramsey,
although this is not easy to see directly. Rather, it follows from \cite[%
Lemma 2.19]{farah_partition_2002} and Proposition \ref{Proposition:layered}
or, alternatively, \cite[Theorem 3.8]{farah_partition_2002}. Once one
observe that any layered action is a Ramsey action of $I_{k}$, one can see
that Theorem \ref{Theorem:monochromatic-tree-action} subsumes \cite[Theorem
3.13]{farah_partition_2002}.

\subsection{A common generalization\label{Subsection:common}}

Let $L$ be a set (alphabet). Suppose that $T$ is a finite rooted tree with
root $r$, and $T^{+}$ is the set of nodes of $T$ different from the root. We
regard the nodes of $T^{+}$ as \emph{variables}. We let $\mathrm{FIN}%
_{L,r}^{T}$ be the set of functions $b:\mathrm{dom}\left( b\right)
\rightarrow L$, where $\mathrm{dom}\left( b\right) $ is a (possibly empty)
subset of $\omega $. For $t\in T^{+}$ let $\mathrm{FIN}_{L,t}^{T}$ be the
set of functions $b:\mathrm{dom}\left( b\right) \rightarrow L\cup T^{+}$
such that $\mathrm{dom}\left( b\right) $ is a finite subset of $\omega $,
and the intersection of the range of $b$ with $T^{+}$ is a nonempty chain
with least element $t$. Let $\mathrm{FIN}_{L}^{T}$ be the set of functions $%
b:\mathrm{dom}\left( b\right) \rightarrow L\cup T^{+}$, w here $\mathrm{dom}%
\left( b\right) $ is a (possibly empty) finite subset of $\omega $. Then $%
\mathrm{FIN}_{L}^{T}$ is endowed with a natural partial semigroup operation.
If $b_{0},b_{1}\in $ $\mathrm{FIN}_{L}^{T}$, then $b_{0}+b_{1}$ is defined
if and only if the domains of $b_{0}$ and $b_{1}$ are disjoint. In such a
case one has that $b_{0}+b_{1}$ is the union of $b_{0}$ and $b_{1}$.

\begin{definition}
\label{Definition:variable-substitution}A \emph{variable substitution map }%
is a \emph{partially defined }function $\sigma $ from a subset $\mathrm{dom}%
\left( \sigma \right) $ of $T^{+}\cup L$ to $T^{+}\cup L$ such that

\begin{itemize}
\item $L\subset \mathrm{dom}\left( \sigma \right) $ and $\sigma |_{L}$ is
the identity function of $L$, and

\item the function $f_{\sigma }:T\rightarrow T$ defined by%
\begin{equation*}
t\mapsto \left\{ 
\begin{array}{ll}
\sigma \left( t\right) & \text{if }t\in \mathrm{dom}\left( \sigma \right) 
\text{ and }\sigma \left( t\right) \in T^{+}\text{,} \\ 
r & \text{otherwise}%
\end{array}%
\right.
\end{equation*}%
is a regressive homomorphism of $T$.
\end{itemize}
\end{definition}

Any variable substitution map $\sigma $ defines an adequate partial
semigroup homomorphism $\tau _{\sigma }$ of $\mathrm{FIN}_{L}^{T}$ as
follows. Suppose that $b\in \mathrm{FIN}_{L}^{T}$. Then the domain of $\tau
_{\sigma }\left( b\right) $ is $\left\{ n\in \mathrm{dom}\left( b\right)
:b\left( n\right) \in \mathrm{dom}\left( \sigma \right) \right\} $, and $%
\tau _{\sigma }\left( b\right) $ is the function $n\mapsto \sigma \left(
b\left( n\right) \right) $. The collection of such maps defines an action of 
$T$ on $\mathrm{FIN}_{L}^{T}$ in the sense of Definition \ref%
{Definition:action-partial}. The \emph{spine }of the variable substitution
map $\tau _{\sigma }$ is the regressive homomorphism $f_{\sigma
}:T\rightarrow T$.

We claim that such an action is Ramsey; see Definition \ref%
{Definition:Ramsey-action} and Theorem \ref%
{Theorem:monochromatic-tree-action}. Fix any minimal idempotent element $\xi
\left( r\right) $ of $\mathrm{FIN}_{L}$. By Lemma \ref{Lemma:minimal-lift}
for every $t_{0}\in T$ of height $1$ there exists a minimal idempotent
element $\xi \left( t_{0}\right) $ of $\gamma \mathrm{FIN}_{L,t_{0}}^{T}$
such that $\xi \left( t_{0}\right) \leq \xi \left( r\right) $. By minimality
of $\xi \left( r\right) $, if $\sigma $ is any variable substitution map
such that $f_{\sigma }\left( t_{0}\right) =r $ one has that $\tau _{\sigma
}\left( \xi \left( t_{0}\right) \right) =\xi \left( r\right) $. Fix now $%
t\in T^{+}$ and consider a predecessor $t_{0}$ of $T$ of height $1$. Let $%
\psi _{t}:\mathrm{FIN}_{L,t_{0}}^{T}\rightarrow \mathrm{FIN}_{L,t}^{T}$ be
the function defined by letting, for $b\in \mathrm{FIN}_{L,t_{0}}^{T}$, $%
\psi _{t}\left( b\right) $ be the function with the same domain as $b$ such
that%
\begin{equation*}
\psi _{t}\left( b\right) \left( n\right) =\left\{ 
\begin{array}{ll}
t & \text{if }b\left( n\right) =t_{0} \\ 
b\left( n\right) & \text{otherwise.}%
\end{array}%
\right.
\end{equation*}%
Since $\psi _{t}$ is an adequate semigroup homomorphism, it extends to a
continuous semigroup homomorphism $\psi _{t}:\gamma \mathrm{FIN}%
_{L,t_{0}}^{T}\rightarrow \gamma \mathrm{FIN}_{L,t}^{T}$. Define now $\xi
\left( t\right) :=\psi _{t}\left( \xi \left( t_{0}\right) \right) $. Observe
that, if $\sigma $ is a variable substitution map, then $\tau _{\sigma
}\left( \xi \left( t\right) \right) =\xi \left( f_{\sigma }\left( t\right)
\right) $. Indeed, if $f_{\sigma }\left( t\right) =s\in T^{+}$ then we have
that $\tau _{\sigma }\left( \xi \left( t\right) \right) =\left( \tau
_{\sigma }\circ \psi _{t}\right) \left( \xi \left( t_{0}\right) \right)
=\psi _{s}\left( \xi \left( t_{0}\right) \right) =\xi \left( s\right) $. If $%
f_{\sigma }\left( t\right) =r$ then fix any variable substitution map $%
\sigma ^{\prime }$ such that $\sigma ^{\prime }\left( t_{0}\right) =\sigma
\left( t\right) $ (where the equality should be interpreted as asserting
that the left hand side is not defined if the right hand side is not
defined). Then we have that $\tau _{\sigma }\left( \xi \left( t\right)
\right) =\left( \tau _{\sigma }\circ \psi _{t}\right) \left( \xi \left(
t_{0}\right) \right) =\tau _{\sigma ^{\prime }}\left( \xi \left(
t_{0}\right) \right) =\xi \left( r\right) =\xi \left( f_{\sigma }\left(
t\right) \right) $. This concludes the proof.

Therefore applying Theorem \ref{Theorem:monochromatic-tree-action} one
obtains the following result, which is a common generalization of Gowers'
theorem for multiple tetris operations and the infinite Hales--Jewett
theorem for located words \cite[Theorem 6.1]{bergelson_partition_1994}.

\begin{theorem}
\label{Theorem:Gowers-Hales-Jewett}Suppose that $T$ is a finite tree, $L$ is
a set, and $c$ is a finite coloring of $\mathrm{FIN}_{L}^{T}$. For any
sequence $\left( d_{n}\right) $ in $\omega $ and sequence $\left( \mathcal{F}%
_{n}\right) $ where $\mathcal{F}_{n}$ is a finite set of variable
substitution maps for $T,L$ as in Definition \ref%
{Definition:variable-substitution}, there exist sequences $\left(
b_{t,n}\right) _{n\in \omega }$ in $\mathrm{FIN}_{L,t}^{T}$ for $t\in T$
such that

\begin{itemize}
\item for every $n\in \omega $ and $t\in T$, $\mathrm{dom}\left(
b_{t,n}\right) >d_{n}$ and $\mathrm{dom}\left( b_{t,n+1}\right) >\mathrm{dom}%
\left( b_{t,n}\right) $;

\item for every $\ell \in \omega $, $n_{0}<\cdots <n_{\ell }\in \omega $, $%
t_{0},\ldots ,t_{\ell }\in T$, variable substitution maps $\sigma
_{0},\ldots ,\sigma _{\ell }$, such that $\left\{ f_{\sigma _{0}}\left(
t_{0}\right) ,\ldots ,f_{\sigma _{\ell }}\left( t_{\ell }\right) \right\} $
is a chain in $T$ least element $t$, one has that the color of%
\begin{equation*}
\tau _{\sigma _{0}}\left( b_{t_{0},n_{0}}\right) +\cdots +\tau _{\sigma
_{\ell }}\left( b_{t_{\ell },n_{\ell }}\right)
\end{equation*}%
depends only on $t$.
\end{itemize}
\end{theorem}

Theorem \ref{Theorem:monochromatic-tree-action} is the particular instance
of Theorem \ref{Theorem:Gowers-Hales-Jewett} where $L=\varnothing $. The
infinite Hales-Jewett theorem for located words \cite[Theorem 6.1]%
{bergelson_partition_1994}---see also \cite[Theorem 2.40]%
{todorcevic_introduction_2010}---is the particular instance of Theorem \ref%
{Theorem:Gowers-Hales-Jewett} where $T$ is the rooted tree with only two
nodes.

More generally from Corollary \ref{Corollary:product-action-tree} one can
deduce the following generalization of Theorem \ref%
{Theorem:Gowers-Hales-Jewett}.

\begin{theorem}
\label{Theorem:product-Gowers-Hales-Jewett}Suppose that $m,k\in \mathbb{N}$
and $\lambda :m\rightarrow k$ is a function. Let $T_{i}$ be a finite rooted
tree and $L_{i}$ be a set for $i\in k$. Fix any coloring $c$ of $\mathrm{FIN}%
_{L_{\lambda \left( 0\right) }}^{T_{\lambda \left( 0\right) }}\times \cdots
\times \mathrm{FIN}_{L_{\lambda \left( m-1\right) }}^{T_{\lambda \left(
m-1\right) }}$. For any sequence $\left( d_{n}\right) $ of elements of $%
\omega $, $i\in k$, sequences $\left( \mathcal{F}_{i,n}\right) $ where $%
\mathcal{F}_{i,n}$ is a finite set of variable substitution maps for $%
T_{i},L_{i}$ as in Definition \ref{Definition:variable-substitution}, for
every and $t\in T_{i}$, there exist sequences $\left( b_{i,t,n}\right)
_{n\in \omega }$ in $\mathrm{FIN}_{L_{i},t}^{T_{i}}$ such that

\begin{itemize}
\item the least element of the domain of $b_{i,t,n}$ is larger than $d_{n}$
and larger than the largest element of the domain of $b_{i,t,n-1}$, and

\item for any $F_{0}<F_{1}<\cdots <F_{m-1}$, nodes $t_{i,d}$ of $T_{i}$,
variable substitution maps $\sigma _{i,d}\in \mathcal{F}_{i,d}$ for $d\in
\omega $ and $i\in k$ such that $\{f_{\sigma _{\lambda \left( s\right)
,d}}\left( t_{\lambda \left( s\right) ,d}\right) :d\in F_{s}\}$ is a chain
in $T_{\lambda \left( s\right) }$ with least element $t_{s}$ for every $s\in
m$, the color of 
\begin{equation*}
\left( \sum_{d\in F_{s}}\tau _{\sigma _{\lambda \left( s\right)
,d}}(b_{f\left( s\right) ,t_{\lambda \left( s\right) },d})\right) _{s\in m}
\end{equation*}%
depends only on $\left( t_{s}\right) _{s\in m}$.
\end{itemize}
\end{theorem}

\section{Actions of trees on filtered semigroups\label%
{Section:action-filtered}}

\subsection{Filtered sets and filtered semigroup\label%
{Subsection:filtered-semigroup}}

By a\emph{\ filtered set} we mean a set $S$ endowed with a distinguished
filter $\mathfrak{F}$. An ultrafilter $\mathcal{U}$ on the filtered set $S$
is \emph{cofinal }if it contains $\mathfrak{F}$. We then let $\beta _{%
\mathfrak{F}}S$ be the closed subsets of $\beta S$ consisting of the cofinal
ultrafilters on $S$. Given a filter $\mathfrak{F}$ on $S$ we define the
corresponding \emph{dual coideal }to be the collection $\mathfrak{F}^{\ast }$
of subsets $A$ of $S$ whose complement does not belong to $\mathfrak{F}$
(equivalently, $A\cap C\neq \varnothing $ for every $C\in \mathfrak{F}$);
see \cite[Section 1]{bergelson_large_2008} and also \cite[Section 2]%
{akin_recurrence_1997}. It is not difficult to verify that this is indeed a
coideal in the sense of \cite[Section 1.1]{todorcevic_introduction_2010}.
Observe that an ultrafilter $\mathcal{U}$ over $S$ contains $\mathfrak{F}$
if and only if it is contained in $\mathfrak{F}^{\ast }$. Furthermore, $%
\mathfrak{F}^{\ast }$ is the union of all the ultrafilters in $\beta _{%
\mathfrak{F}}S$, while $\mathfrak{F}$ is the intersection of all the
ultrafilters in $\beta _{\mathfrak{F}}S$; see \cite[Theorem 3.11]%
{hindman_algebra_2012}.

Suppose now that $\left( S,+\right) $ is a partial semigroup. Let $C$ be a
subset of $S$ and $a$ be an element of $S$. Using notation from \cite%
{farah_partition_2002} we denote by $-a+C$ the set of elements $b$ of $S$
such that $a+b$ is defined and belongs to $C$.

\begin{definition}
\label{Definition:filterd-semigroup}A filtered semigroup is a triple $\left(
S,\mathfrak{F},+\right) $ where $\left( S,+\right) $ is a partial semigroup, 
$\mathfrak{F}$ is a filter on $S$ such that

\begin{itemize}
\item for every $x\in S$, $\mathfrak{F}$ contains $\{y\in S:x+y$ is defined$%
\}$;

\item for any $C\in \mathfrak{F}$ and for any $B\in \mathfrak{F}^{\ast }$
there exists a finite subset $F$ of $B$ such that $\bigcup_{a\in F}\left(
-a+C\right) \in \mathfrak{F}$.
\end{itemize}
\end{definition}

\begin{proposition}
\label{Propositon:filterd-semigroup}Suppose that $\left( S,+\right) $ is an
adequate partial semigroup, and $\mathfrak{F}$ is a filter on $S$. Then $%
\left( S,\mathfrak{F},+\right) $ is a filtered semigroup if and only if $%
\beta _{\mathfrak{F}}S$ is a closed subsemigroup of $\gamma S$.
\end{proposition}

\begin{proof}
Suppose that $\left( S,\mathfrak{F},+\right) $ is a filtered semigroup.
Since $\mathfrak{F}$ contains $\left\{ y\in S:x+y\text{ is defined}\right\} $
for every $x\in S$, it follows that $\beta _{\mathfrak{F}}S$ is contained in 
$\gamma S$. We want to prove that $\beta _{\mathfrak{F}}S$ is a subsemigroup
of $\gamma S$. Let $\mathcal{U},\mathcal{V}$ be ultrafilters in $\beta _{%
\mathfrak{F}}S$. Suppose by contradiction that $\mathcal{U}+\mathcal{V}%
\notin \beta _{\mathfrak{F}}S$. Then there exists $C\in \mathfrak{F}$ such
that $S\diagdown C\in \mathcal{U}+\mathcal{V}$. Therefore we have that $%
B:=\left\{ x\in S:\mathcal{V}y\text{, }x+y\in S\diagdown C\right\} \in 
\mathcal{U}\subset \mathfrak{F}^{\ast }$. Suppose that $F$ is a finite
subset of $B$. Then we have that $\mathcal{V}y$, $\forall a\in F$, $a+y\in
S\diagdown C$. Therefore $\left\{ y\in Y:\forall a\in F\text{, }a+y\in
S\diagdown C\right\} \in \mathcal{V}\subset \mathfrak{F}^{\ast }$. Therefore 
$\left\{ y\in Y:\exists a\in F\text{, }a+y\in C\right\} \notin \mathfrak{F}$%
. This contradicts the assumption that $\left( S,\mathfrak{F},+\right) $ is
a filtered semigroup.

Conversely, suppose that $\beta _{\mathfrak{F}}S$ is a subsemigroup of $%
\gamma S$. Thus we have that 
\begin{equation*}
\mathfrak{F}=\bigcap \beta _{\mathfrak{F}}S\supset \bigcap \gamma S\supset
\left\{ \left\{ y\in S:x+y\text{ is defined}\right\} :x\in S\right\} \text{.}
\end{equation*}%
This shows that $\mathfrak{F}$ satisfies the first condition in the
definition of filtered semigroup. We now want to verify the second
condition. Suppose by contradiction that there exist $C\in \mathfrak{F}$ and 
$B\in \mathfrak{F}^{\ast }$ such that for any finite subset $F$ of $B$ one
has that $\bigcup_{\in F}\left( -a+C\right) \notin \mathfrak{F}$. Therefore
we have that $\bigcap_{x\in F}\left( -a+S\diagdown C\right) \in \mathfrak{F}%
^{\ast }$ for every finite subset $F$ of $B$. Therefore by \cite[Theorem 3.11%
]{hindman_algebra_2012} there exists an ultrafilter $\mathcal{V}$ over $S$
such that $\left\{ -a+S\diagdown C:a\in B\right\} \subset \mathcal{V}\subset 
\mathfrak{F}^{\ast }$, as well as an ultrafilter $\mathcal{U}$ over $S$ such
that $B\in \mathcal{U}\subset \mathfrak{F}^{\ast }$. By the choice of $%
\mathcal{U}$ and $\mathcal{V}$ we have that $C\notin \mathcal{U}+\mathcal{V}$%
. Therefore $\mathcal{U},\mathcal{V}$ are elements of $\beta _{\mathfrak{F}%
}S $ such that $\mathcal{U}+\mathcal{V}\notin \beta _{\mathfrak{F}}S$,
contradicting our assumption that $\beta _{\mathfrak{F}}S$ is a subsemigroup
of $\gamma S$.
\end{proof}

Theorem 4.28 of \cite{hindman_algebra_2012} is the particular instance of
Proposition \ref{Propositon:filterd-semigroup} in the case when the
operation in $S$ is everywhere defined and $\mathfrak{F}$ is the filter of
cofinite subsets of $S$.

Suppose that $\left( S,\mathfrak{F}\right) $ and $\left( T,\mathfrak{G}%
\right) $ are filtered semigroups, and $\sigma :S\rightarrow T$ is a
function. We say that $\sigma $ is a \emph{filtered semigroup homomorphism }%
if $\sigma $ is a partial semigroup homomorphism and $\sigma \left( 
\mathfrak{F}\right) =\mathfrak{G}$, where $\sigma \left( \mathfrak{G}\right) 
$ is the collection of subsets of $T$ whose inverse image under $\sigma $
belongs to $\mathfrak{F}$. This implies that the canonical continuous
extension $\sigma :\beta S\rightarrow \beta T$ restricts to a continuous
semigroup homomorphism from $\beta _{\mathfrak{F}}S$ to $\beta _{\mathfrak{G}%
}T$. We say that $\left( T,\mathfrak{G}\right) $ is a \emph{filtered
subsemigroup }of $\left( S,\mathfrak{F}\right) $ if $T$ is a partial
subsemigroup of $S$ and the inclusion map is a filtered semigroup
homomorphism. In other words, the distinguished filter $\mathfrak{G}$ on $T$
is just the trace $\mathfrak{F}|_{T}:=\left\{ A\cap T:A\in \mathfrak{F}%
\right\} $ of $\mathfrak{F}$ on $T$. This allows one to identify $\beta _{%
\mathfrak{G}}T$ with a closed subsemigroup of $\beta _{\mathfrak{F}}S$. In
this situation we denote $\beta _{\mathfrak{F}|_{T}}T$ by $\beta _{\mathfrak{%
F}}T$.

Suppose that $\left( S,+,\mathfrak{F}\right) $ is a filtered semigroup. We
denote by $\mathcal{S}\left( S,+,\mathfrak{F}\right) $ the set of filtered
subsemigroups of $S$. The set $\mathcal{S}\left( S,+,\mathfrak{F}\right) $
is endowed with a canonical ordering obtained by setting $S_{0}\leq S_{1}$
if and only if, for every $s_{0}\in S_{0}$ and $s_{1}\in S_{1}$, $%
s_{0}+s_{1} $ and $s_{1}+s_{0}$ belong to $S_{0}$ whenever they are defined.
The map $S_{0}\mapsto \beta _{\mathfrak{F}}S_{0}$ defines an
order-preserving function from the space $\mathcal{S}\left( S,+,\mathfrak{F}%
\right) $ of filtered subsemigroups of $S$ to the space $\mathcal{S}\left(
\beta _{\mathfrak{F}}S\right) $ of closed subsemigroups of $\beta _{%
\mathfrak{F}}S$.

\begin{example}
Suppose that $\left( S,+\right) $ is a semigroup and $\left( x_{n}\right) $
is a sequence in $S$. Consider then the filter $\mathfrak{F}$ on $S$
generated by sets of the form $\mathrm{FS}\left( x_{n}\right) _{n\geq
r}=\left\{ x_{n_{0}}+x_{n_{1}}+\cdots +x_{n_{\ell }}:\ell \in \omega ,r\leq
n_{0}<\cdots <n_{\ell }\in \omega \right\} $ for $r\in \omega $. Then $%
\left( S,+,\mathfrak{F}\right) $ is a filtered semigroup. In this case $%
\beta _{\mathfrak{F}}S$ is the closed subsemigroup of $\beta S$ consisting
of ultrafilters on $S$ that contain $\mathrm{FS}\left( x_{n}\right) _{n\geq
r}$ for every $r\in \omega $.
\end{example}

\begin{example}
Suppose that $S$ is a semigroup, and $\mathfrak{F}$ is the filter of
cofinite sets. Then $\beta _{\mathfrak{F}}S$ is the set of nonprincipal
ultrafilters over $S$. If $S$ is either left or right cancellative, then $%
\left( S,\mathfrak{F},+\right) $ is a filtered semigroup; see also \cite[%
Corollary 4.29]{hindman_algebra_2012}.
\end{example}

\begin{remark}
Any filtered semigroup $S$ can be seen as a filtered semigroup \emph{where
moreover the operation is everywhere defined }by adding an extra \emph{%
absorbing element}. Let $\left( S,+,\mathfrak{F}\right) $ be a filtered
partial semigroup such that $S$ does not contain the symbol $0$. Then one
can consider the semigroup $S_{0}$ by setting $0+x=x+0=0$ for any $x\in
S_{0} $ and $x+y=0$ whenever $x,y\in S$ and $x+y$ is not defined in $S$.
Then one can consider the filter $\mathfrak{F}_{0}$ of subsets of $S_{0}$
generated $\mathfrak{F}$. This gives a filtered semigroup $\left( S_{0},+,%
\mathfrak{F}_{0}\right) $ where moreover the operation on $S_{0}$ is
everywhere defined. The compact right topological semigroup $\beta _{%
\mathfrak{F}_{0}}S_{0}$ is canonically isomorphic to $\beta _{\mathfrak{F}}S$%
.

If $S,T$ are filtered semigroups and $\sigma :S\rightarrow T$ is a filtered
semigroup homomorphism, then one can canonically extend $\sigma $ to a
filtered semigroup homomorphism $\sigma _{0}:S_{0}\rightarrow T_{0}$ such
that $\sigma \left( 0\right) =0$.
\end{remark}

\subsection{Actions of ordered sets on filtered semigroups}

Suppose that $\mathbb{P}$ is an ordered set, and $\left( S,+,\mathfrak{F}%
\right) $ is a filtered semigroup. We denote by $\mathrm{End}\left( S,+,%
\mathfrak{F}\right) $ the space of filtered semigroup homomorphisms $\tau
:S\rightarrow S$. Observe that $\mathrm{\mathrm{En}d}\left( S,+,\mathfrak{F}%
\right) $ is a semigroup with respect to composition.

\begin{definition}
\label{Definition:action-filtered}An \emph{action }$\alpha $ of $\mathbb{P}$
on $\left( S,+,\mathfrak{F}\right) $ is given by

\begin{itemize}
\item an order-preserving function $\mathbb{P}\rightarrow \mathcal{S}\left(
S,+,\mathfrak{F}\right) $, $t\mapsto S_{t}$, and

\item a subsemigroup $\mathcal{F}_{\alpha }\subset \mathrm{End}\left( S,+,%
\mathfrak{F}\right) $,
\end{itemize}

such that such that for every $\tau \in \mathcal{F}_{\alpha }$ there exists
a function $f_{\tau }:\mathbb{P}\rightarrow \mathbb{P}$---which we call the 
\emph{spine} of $\tau $---such that $\tau $ maps $S_{t}$ to $S_{f_{\tau
}\left( t\right) }$ for every $t\in \mathbb{P}$, and such that $\tau \left(
s\right) =s$ for any $s\in S_{t}$ and $t\in T$ such that $f_{\tau }\left(
t\right) =t$.
\end{definition}

As in the case of actions on adequate partial semigroups, an action $\alpha $
of $\mathbb{P}$ on $\left( S,+,\mathfrak{F}\right) $ induces an action---in
the sense of Definition \ref{Definition:action-compact}---of $\mathbb{P}$ on
the compact right topological semigroup $X=\beta _{\mathfrak{F}}S$, which we
still denote by $\alpha $. Such an action is obtained by setting $%
X_{t}:=\beta _{\mathfrak{F}}S_{t}$ and $\tau :\beta _{\mathfrak{F}%
}S\rightarrow \beta _{\mathfrak{F}}S$ to be the unique extension of $\tau $
to a continuous partial semigroup endomomorphism of $\beta _{\mathfrak{F}}S$
for every $\tau \in \mathcal{F}_{\alpha }$. Consistently with the notation
used in Subsection \ref{Subsection:action-compact-semigroup}, we denote by $%
(\beta _{\mathfrak{F}}S)_{\alpha }$ the set of functions $\xi :\mathbb{P}%
\rightarrow \beta _{\mathfrak{F}}S$ such that $\xi \left( t\right) \in \beta
_{\mathfrak{F}}S_{t}$ and $\tau \circ \xi =\xi \circ f_{\tau }$ for every $%
\tau \in \mathcal{F}_{\alpha }$. An order-preserving idempotent in $\left(
\beta _{\mathfrak{F}}S\right) _{\alpha }$ is an element $\xi $ of $\left(
\beta _{\mathfrak{F}}S\right) _{\alpha }$ such that $\xi \left( t\right) $
is an idempotent element of $\beta _{\mathfrak{F}}S_{t}$ and $\xi \left(
t\right) \leq \xi \left( t_{0}\right) $ whenever $t,t_{0}\in \mathbb{P}$ are
such that $t\leq t_{0}$.

The same proof as Theorem \ref{Theorem:monochromatic-action} gives the
following.

\begin{theorem}
\label{Theorem:product-action-filtered}Suppose that $k,m\in \mathbb{N}$ and $%
\lambda :m\rightarrow k$ is a function.\ For $i\in k$ let $\left(
S_{i},+_{i},\mathfrak{F}_{i}\right) $ be a filtered semigroup, $\mathbb{P}%
_{i}$ a finite ordered set, and $\alpha _{i}$ an action of $\mathbb{P}_{i}$
on $\left( S_{i},+_{i},\mathfrak{F}_{i}\right) $. Suppose that $\xi _{i}\in
\left( \gamma S_{i}\right) _{\alpha _{i}}$ is an order-preserving idempotent
for $i\in k$. Set $S:=S_{\lambda \left( 0\right) }\times \cdots \times
S_{\lambda \left( m-1\right) }$ and suppose that $c$ is a finite coloring of 
$S$. Extend $c$ canonically to a coloring of $\beta S$. Fix a sequence $%
(\psi _{i,n}^{(\mathfrak{F})})$ of functions $\psi _{i,n}:(S_{i}^{\mathbb{P}%
_{i}})^{n}\rightarrow \mathfrak{F}_{i}$, and sequence $(\psi _{i,n}^{(%
\mathcal{F})})$ of functions $\psi _{i,n}:(S_{i}^{\mathbb{P}%
_{i}})^{n}\rightarrow \left[ \mathfrak{F}_{i}\right] ^{<\aleph _{0}}$. Then
there exist sequences $\left( x_{i,n}\right) _{n\in \omega }$ of functions $%
x_{i,n}:\mathbb{P}_{i}\rightarrow S_{i}$ such that

\begin{itemize}
\item $x_{i,n}\left( t\right) \in S_{i,t}\cap \psi _{i,n}^{(\mathfrak{F}%
)}\left( x_{i,0},\ldots ,x_{i,n-1}\right) $ for every $n\in \omega $, $i\in
k $, and $t\in \mathbb{P}_{i}$; and

\item for any $\ell \in \omega $, $i\in k$, finite subsets $%
F_{0}<F_{1}<\cdots <F_{m-1}$ of $\omega $, sequences $\left( t_{i,n}\right)
_{n\in \omega }$ in $\mathbb{P}_{i}$ and $\left( \tau _{i,n}\right) _{n\in
\omega }$ in $\mathcal{F}_{\alpha _{i}}$ such that $\tau _{i,n}\in \psi
_{i,n}^{(\mathcal{F})}\left( x_{i,0},\ldots ,x_{i,n-1}\right) $ for every $%
n\in \omega $, if $\{f_{\tau _{\lambda \left( s\right) ,d}}\left( t_{\lambda
\left( s\right) ,d}\right) :d\in F_{s}\}$ is a chain in $\mathbb{P}_{\lambda
\left( s\right) }$ with least element $t_{s}$ for every $s\in m$, then the
color of%
\begin{equation*}
\left( \sum\nolimits_{d\in F_{s}}^{\lambda \left( s\right) }\tau _{\lambda
(s),d}\left( x_{\lambda (s),d}\left( t_{\lambda \left( s\right) ,d}\right)
\right) \right) _{s\in m}
\end{equation*}%
is the same as the color of $\xi _{\lambda (0)}\left( t_{0}\right) \otimes
\xi _{\lambda \left( 1\right) }\left( t_{1}\right) \otimes \cdots \otimes
\xi _{\lambda \left( m-1\right) }\left( t_{m-1}\right) $.
\end{itemize}
\end{theorem}

Suppose now that $T$ is a finite rooted tree endowed with the canonical tree
ordering.

\begin{definition}
\label{Definition:Ramsey-filtered}Suppose that $\alpha $ is an action of $T$
on an filtered semigroup $\left( S,+,\mathfrak{F}\right) $ given by a
semigroup $\mathcal{F}_{\alpha }\subset \mathrm{End}\left( S,+,\mathfrak{F}%
\right) $. We say that $\alpha $ is \emph{Ramsey }if, for any $\tau \in 
\mathcal{F}_{\alpha }$, the corresponding spine $f_{\tau }:T\rightarrow T$
is a regressive homomorphis, and for any $C\in \mathfrak{F}$, for any finite
coloring $c$ of $S$, and for any finite subset $\mathcal{F}_{0}$ of $%
\mathcal{F}_{\alpha }$, there exists a function $x:T\rightarrow S$ such
that, for any $\tau \in \mathcal{F}_{0}$ and $t\in T$, $x\left( t\right) \in
S_{t}\cap C$ and the color of $\tau \left( x\left( t\right) \right) $
depends only on $f_{\tau }\left( t\right) $.
\end{definition}

The same proofs as Theorem \ref{Theorem:monochromatic-tree-action} and
Theorem \ref{Theorem:product-action-filtered} give the following.

\begin{theorem}
\label{Theorem:monochromatic-tree-filtered}Suppose that $\alpha $ is an
action of a finite rooted tree on a filtered semigroup $\left( S,+,\mathfrak{%
F}\right) $ such that, for every $\tau \in \mathcal{F}_{\alpha }$, the
corresponding spine $f_{\tau }$ is a regressive homomorphism. The following
statements are equivalent:

\begin{enumerate}
\item $\alpha $ is Ramsey;

\item the action of $T$ on $\beta _{\mathfrak{F}}S$ induced by $\alpha $ is
Ramsey;

\item for any finite coloring $c$ of $S$, sequence $\left( \psi _{n}\right) $
of functions $\psi _{n}^{(\mathfrak{F})}:\left( S^{T}\right) ^{n}\rightarrow 
\mathfrak{F}$, and sequence $(\psi _{n}^{(\mathcal{F})})$ of functions $\psi
_{n}^{(\mathcal{F})}:\left( S^{T}\right) ^{n}\rightarrow \left[ \mathcal{F}%
\right] ^{<\aleph _{0}}$, there exist functions $x_{n}:T\rightarrow S$ such
that

\begin{itemize}
\item $x_{n}\left( t\right) \in S_{t}\cap \psi _{n}^{(\mathfrak{F})}\left(
x_{0},\ldots ,x_{n-1}\right) $ for every $n\in \omega $ and $t\in T$; and

\item for any $\ell \in \omega $, $n_{0}<n_{1}<\cdots <n_{\ell }\in \omega $%
, $t_{i}\in T$ for $i\leq \ell $, and $\tau _{i}\in \mathcal{F}_{n_{i}}$ for 
$i\leq \ell $, if $\left\{ f_{\tau _{i}}\left( t_{i}\right) :i\leq \ell
\right\} $ is a chain in $T$ with least element $t$, then the color of $\tau
_{0}\left( x_{n_{0}}\left( t_{0}\right) \right) +\cdots +\tau _{\ell }\left(
x_{n_{\ell }}\left( t_{\ell }\right) \right) $ depends only on $t$.
\end{itemize}
\end{enumerate}
\end{theorem}

\begin{theorem}
\label{Theorem:product-tree-filtered}Suppose that $k,m\in \mathbb{N}$ and $%
\lambda :m\rightarrow k$ is a function.\ For $i\in k$ let $\left(
S_{i},+_{i},\mathfrak{F}_{i}\right) $ be a filtered semigroup, $T_{i}$ a
finite rooted tree set, and $\alpha _{i}$ a Ramsey action of $T_{i}$ on $%
\left( S_{i},+_{i},\mathfrak{F}_{i}\right) $. Set $S:=S_{\lambda \left(
0\right) }\times \cdots \times S_{\lambda \left( m-1\right) }$ and suppose
that $c$ is a finite coloring of $S$. Fix a sequence $(\psi _{i,n}^{(%
\mathfrak{F})})$ of functions $\psi _{i,n}^{(\mathfrak{F}%
)}:(S_{i}^{T_{i}})^{n}\rightarrow \mathfrak{F}_{i}$, and sequence $(\psi
_{i,n}^{(\mathcal{F})})$ of functions $\psi _{i,n}^{(\mathcal{F}%
)}:(S_{i}^{T_{i}})^{n}\rightarrow \left[ \mathcal{F}_{i}\right] ^{<\aleph
_{0}}$. Then there exist sequences $\left( x_{i,n}\right) _{n\in \omega }$
of functions $x_{i,n}:T_{i}\rightarrow S_{i}$ such that

\begin{itemize}
\item $x_{i,n}\left( t\right) \in S_{i,t}\cap \psi _{i,n}^{(\mathfrak{F}%
)}\left( x_{i,0},\ldots ,x_{i,n-1}\right) $ for every $n\in \omega $, $i\in
k $, and $t\in T_{i}$; and

\item for any $\ell \in \omega $, $i\in k$, $F_{0}<F_{1}<\cdots <F_{m-1}$,
sequences $\left( t_{i,n}\right) _{n\in \omega }$ in $T_{i}$ and $\left(
\tau _{i,n}\right) _{n\in \omega }$ in $\mathcal{F}_{\alpha _{i}}$ such that 
$\tau _{i,n}\in \psi _{i,n}^{(\mathcal{F})}\left( x_{0},\ldots
,x_{n-1}\right) $ for every $n\in \omega $, if $\{f_{\tau _{\lambda \left(
s\right) ,d}}\left( t_{\lambda \left( s\right) ,d}\right) :d\in F_{s}\}$ is
a chain in $T_{\lambda \left( s\right) }$ with least element $t_{s}$ for
every $s\in m$, then the color of%
\begin{equation*}
\left( \sum\nolimits_{d\in F_{s}}^{\lambda \left( s\right) }\tau _{\lambda
(s),d}\left( x_{\lambda (s),d}\left( t_{\lambda (s),d}\right) \right)
\right) _{s\in m}
\end{equation*}%
depends only on $\left( t_{s}\right) _{s\in m}$.
\end{itemize}
\end{theorem}

\subsection{The Hales-Jewett theorem for nonlocated words}

Suppose that $L$ is a set (\emph{alphabet}). Let $\mathrm{W}_{L}$ be the set
of finite strings of elements of $L$ (\emph{words}).\ Let also $v$ be a
symbol not in $L$ (\emph{variable}), and $\mathrm{W}_{Lv}$ be the set of
finite strings of elements of $L$ where $v$ appears (\emph{variable words}).
Then $S:=\mathrm{W}_{L}\cup \mathrm{W}_{Lv}$ is a cancellative semigroup
with respect to the operation $+$ of concatenation. Thus $\left( S,\mathfrak{%
F},+\right) $ is a filtered semigroup, where $\mathfrak{F}$ is the filter of
cofinite subsets of $S$.

Any element $a\in L$ defines a variable substitution map $\tau
_{a}:S\rightarrow S$ mapping any word $w$ to the word obtained from $w$ by
replacing every occurrence of the variable $v$ (if any) by $a$. This defines
a Ramsey action of the rooted tree with two nodes on $\left( S,\mathfrak{F}%
,+\right) $. It follows from Proposition \ref{Proposition:layered} that such
an action is Ramsey. Therefore Theorem \ref{Theorem:product-action-filtered}
in this case recovers the infinite Hales-Jewett theorem (for nonlocated
words) \cite[Theorem 1.1]{bergelson_partition_1994}---see also \cite[Theorem
2.35]{todorcevic_introduction_2010}.

One can also consider a generalization of such a result for sets of
variables indexed by a tree similar to Theorem \ref%
{Theorem:Gowers-Hales-Jewett}. Suppose that $T$ is a finite rooted tree with
root $r$, and let $T^{+}$ be the set of nodes of $T$ different from the
root. We think of the elements of $T^{+}$ as variables. Let also $L$ be a
set. Define \textrm{W}$_{L,r}^{T}$ to be the set of words with symbols from $%
L$. For $t\in T^{+}$ let $\mathrm{W}_{L,t}^{T}$ be the set of words $b$ with
symbols from $L\cup T^{+}$ such that the \emph{variables }that appear in $b$
form a nonempty chain in $T$ with least element $t$. We then let $\mathrm{W}%
_{L}^{T}$ be the set of all words in $L\cup T^{+}$. This is a cancellative
semigroup with respect to the concatenation operation $+$. We regard $\left( 
\mathrm{W}_{L}^{T},+\right) $ as a \emph{filtered }semigroup endowed with
the filter $\mathfrak{F}$ of cofinite subsets of \textrm{W}$_{L}^{T}$.

\begin{definition}
\label{Definition:variable-substitution-nonlocated}A \emph{variable
substitution map }is a function $\sigma $ from $T^{+}\cup L$ to $T^{+}\cup L$
such that

\begin{itemize}
\item $\sigma |_{L}$ is the identity function of $L$, and

\item the function $f_{\sigma }:T\rightarrow T$ defined by%
\begin{equation*}
t\mapsto \left\{ 
\begin{array}{ll}
\sigma \left( t\right) & \text{if }t\in T^{+}\text{ and }\sigma \left(
t\right) \in T^{+}\text{,} \\ 
r & \text{otherwise}%
\end{array}%
\right.
\end{equation*}%
is a regressive homomorphism of $T$.
\end{itemize}
\end{definition}

Any variable substitution map $\sigma $ defines a filtered semigroup
homomorphism $\tau _{\sigma }$ of $\left( \mathrm{W}_{L}^{T},+,\mathfrak{F}%
\right) $ defined as follows. For every $b\in \mathrm{W}_{L}^{T}$, $\tau
_{\sigma }\left( b\right) $ is the word obtained from $b$ replacing every
occurrence of $x$ with $\sigma \left( x\right) $ for every $x\in T^{+}\cup L$%
. The same argument as in Subsection \ref{Subsection:HJ} shows that such an
action is Ramsey; see Definition \ref{Definition:Ramsey-filtered}.

Therefore applying Theorem \ref{Theorem:product-tree-filtered} one obtains
the following result, which is a generalization of the infinite
Hales--Jewett theorem for nonlocated words \cite[Theorem 1.1]%
{bergelson_partition_1994}.

\begin{theorem}
\label{Theorem:nonlocated-Gowers-Hales-Jewett}Suppose that $T$ is a finite
tree, $L$ is a set, and $c$ is a finite coloring of $\mathrm{W}_{L}^{T}$.
For any sequence $\left( d_{n}\right) $ in $\omega $ and sequence $\left( 
\mathcal{F}_{n}\right) $ where $\mathcal{F}_{n}$ is a finite set of variable
substitution maps for $T,L$ as in Definition \ref%
{Definition:variable-substitution-nonlocated}, there exist sequences $\left(
b_{t,n}\right) _{n\in \omega }$ in $\mathrm{W}_{L,t}^{T}$ for $t\in T$ such
that

\begin{itemize}
\item for every $n\in \omega $ and $t\in T$, the length of $b_{t,n}$ is at
least $d_{n}$, and

\item for every $\ell \in \omega $, $n_{0}<\cdots <n_{\ell }\in \omega $, $%
t_{0},\ldots ,t_{\ell }\in T$, variable substitution maps $\sigma _{j}\in 
\mathcal{F}_{j}$ for $j\leq \ell $ such that $\left\{ f_{\sigma _{0}}\left(
t_{0}\right) ,\ldots ,f_{\sigma _{\ell }}\left( t_{\ell }\right) \right\} $
is a chain in $T$ with least element $t\in T$ one has that the color of%
\begin{equation*}
\tau _{\sigma _{0}}\left( b_{t_{0},n_{0}}\right) +\cdots +\tau _{\sigma
_{\ell }}\left( b_{t_{\ell },n_{\ell }}\right)
\end{equation*}%
depends only on $t$.
\end{itemize}
\end{theorem}

The infinite Hales-Jewett theorem for nonlocated words \cite[Theorem 1.1]%
{bergelson_partition_1994} is the particular instance of Theorem \ref%
{Theorem:Gowers-Hales-Jewett} where $T$ is the rooted tree with only two
nodes.

More generally from Corollary \ref{Corollary:product-action-tree} one can
deduce the following generalization of Theorem \ref%
{Theorem:Gowers-Hales-Jewett}.

\begin{theorem}
\label{Theorem:product-Gowers-Hales-Jewett-nonlocated}Suppose that $m,k\in 
\mathbb{N}$ and $\lambda :m\rightarrow k$ is a function. Let $T_{i}$ be a
finite rooted tree and $L_{i}$ be a set for $i\in k$. Fix a finite coloring $%
c$ of $\mathrm{W}_{L_{\lambda \left( 0\right) }}^{T_{\lambda \left( 0\right)
}}\times \cdots \times \mathrm{W}_{L_{\lambda \left( m-1\right)
}}^{T_{\lambda \left( m-1\right) }}$. For any sequence $\left( d_{n}\right) $
of elements of $\omega $, $i\in k$, sequence $\left( \mathcal{F}%
_{i,n}\right) $ of finite sets of variable substitution maps for $%
T_{i},L_{i} $ in the sense of Definition \ref%
{Definition:variable-substitution-nonlocated}, and for every $t\in T_{i}$,
there exist sequences $\left( b_{i,t,n}\right) _{n\in \omega }$ in $\mathrm{%
FIN}_{L_{i},t}^{T_{i}}$ such that

\begin{itemize}
\item the length of $b_{i,t,n}$ is at least $d_{n}$,

\item for any $F_{0}<F_{1}<\cdots <F_{m-1}$, nodes $t_{i,d}$ of $T_{i}$ and
variable substitution maps $\sigma _{i,d}\in \mathcal{F}_{i,d}$ for $d\in
\omega $ and $i\in k$ such that $\{f_{\sigma _{\lambda \left( s\right)
,d}}\left( t_{\lambda \left( s\right) ,d}\right) :d\in F_{s}\}$ is a chain
in $T_{\lambda \left( s\right) }$ with least element $t_{s}$ for every $s\in
m$, the color of 
\begin{equation*}
\left( \sum_{d\in F_{s}}\tau _{\sigma _{\lambda \left( s\right)
,d}}(b_{f\left( s\right) ,t_{\lambda \left( s\right) },d})\right) _{s\in m}
\end{equation*}%
depends only on $\left( t_{s}\right) _{s\in m}$.
\end{itemize}
\end{theorem}

\section{A polynomial Gowers' Ramsey theorem\label{Section:polynomial}}

\subsection{Extended polynomials\label{Subsection:extended}}

We follow the approach and notation of \cite{bergelson_polynomial_2014}.
Suppose that $\left( S,\mathfrak{F}\right) $ is a filtered set. Let $%
\mathfrak{S}$ be a collection of partial semigroups operations on $S$ such
that, for every $+\in \mathfrak{S}$, $\left( S,+,\mathfrak{F}\right) $ is a
filtered semigroup. The space $\beta _{\mathfrak{F}}S$ of cofinite
ultrafilters over $\left( S,\mathfrak{F}\right) $ is then a semigroup with
respect to the canonical extension to $\beta S$ of any of the operations in $%
\mathfrak{S}$. In the following if $a\in S$ and $\mathfrak{G}$ is a filter
over $S$ we let $a+\mathfrak{G}$ be the image of $\mathfrak{G}$ under the
left translation map $x\mapsto a+x$. Explicitly, $C\in a+\mathfrak{G}$ if
and only if $-a+C=\left\{ x\in S:a+x\in C\right\} \in \mathfrak{G}$.
Similarly we define $\mathfrak{G}+a$ in terms of the right translation map $%
x\mapsto x+a$.

One can define as in \cite[Definition 3.1]{bergelson_polynomial_2014} the
set $\mathfrak{P}$ of \emph{extended polynomials }in the variables $%
x_{0},x_{1},\ldots $\emph{\ }corresponding to the set $\mathfrak{S}$ by
induction on the degree as follows:

\begin{enumerate}
\item $x_{n}$ is an extended polynomial for every $n\in \omega $;

\item if $+\in \mathfrak{S}$ and $p\left( x_{0},\ldots ,x_{n-1}\right) $ and 
$q\left( x_{n},\ldots ,x_{n+m-1}\right) $ are extended polynomials, then $%
p\left( x_{0},\ldots ,x_{n-1}\right) +q\left( x_{n},\ldots ,x_{n+m-1}\right) 
$ is an extended polynomial;

\item if $+\in \mathfrak{S}$, $p\left( x_{0},\ldots ,x_{n-1}\right) $ is an
extended polynomial, and $a\in \mathfrak{S}$ is such that $\mathfrak{F}%
+a\supset \mathfrak{F}$, then $p\left( x_{0},\ldots ,x_{n-1}\right) +a$ is
an extended polynomial;

\item if $+\in \mathfrak{S}$, $p\left( x_{0},\ldots ,x_{n-1}\right) $ is an
extended polynomial, and $a\in \mathfrak{S}$ is such that $a+\mathfrak{F}%
\supset \mathfrak{F}$ then $a+p\left( x_{0},\ldots ,x_{n-1}\right) $ is an
extended polynomial.
\end{enumerate}

Any extended polynomial $p\left( x_{0},\ldots ,x_{n-1}\right) $ defines a
polynomial mapping $f_{p}:S^{n}\rightarrow S$ in the obvious way.\ One can
then consider its canonical extension to a continuous function $f_{p}:\beta
\left( S^{n}\right) \rightarrow \beta S$. One can also evaluate a polynomial 
$p\left( x_{0},\ldots ,x_{n-1}\right) $ at a tuple $\left( \mathcal{U}%
_{0},\ldots ,\mathcal{U}_{n-1}\right) $ of elements of $\beta _{\mathfrak{F}%
}S$ by interpreting the operations in $\mathfrak{S}$ as their canonical
extensions to right topological semigroup operations on $\beta _{\mathfrak{F}%
}S$. The following proposition is the analog of Theorem \cite[Theorem 3.2]%
{bergelson_polynomial_2014} in this context. The proof is entirely
analogous, and it is presented here for convenience of the reader.

\begin{proposition}
\label{Proposition:extended-polynomial}Suppose that $p\left( x_{0},\ldots
,x_{n-1}\right) $ is an extended polynomial, and $f_{p}:S^{n}\rightarrow S$
is the corresponding polynomial mapping. Then for every $\mathcal{U}%
_{0},\ldots ,\mathcal{U}_{n-1}\in \beta _{\mathfrak{F}}S$ we have that $%
p\left( \mathcal{U}_{0},\ldots ,\mathcal{U}_{n-1}\right) \in \beta _{%
\mathfrak{F}}S$ is equal to $f_{p}\left( \mathcal{U}_{0}\otimes \cdots
\otimes \mathcal{U}_{n-1}\right) $.
\end{proposition}

\begin{proof}
In the course of the proof, we denote tuples of variables $x_{0},\ldots
,x_{n-1}$ and $y_{0},\ldots ,y_{m-1}$ by $\overline{x}$ and $\overline{y}$
respectively. The proof is naturally by induction on the complexity of the
given extended polynomial. Clearly the conclusion is true for the polynomial 
$x_{n}$.

Suppose that the conclusion is true for $p\left( \overline{x}\right) $ and $%
q\left( \overline{y}\right) $. Fix $+\in \mathfrak{S}$ and set $r\left( 
\overline{x},\overline{y}\right) :=p\left( \overline{x}\right) +q\left( 
\overline{y}\right) $. Let $\mathcal{U}_{0},\ldots ,\mathcal{U}_{n-1}$ and $%
\mathcal{V}_{0},\ldots ,\mathcal{V}_{m-1}$ be elements of $\beta _{\mathfrak{%
F}}\left( S\right) $. Set $\mathcal{U}:=\mathcal{U}_{0}\otimes \cdots
\otimes \mathcal{U}_{n-1}$ and $\mathcal{V}:=\mathcal{V}_{0}\otimes \cdots
\otimes \mathcal{V}_{m-1}$. Then we have that%
\begin{equation*}
r\left( \mathcal{U}_{0},\ldots ,\mathcal{U}_{n-1},\mathcal{V}_{0},\ldots ,%
\mathcal{V}_{m-1}\right) =p\left( \mathcal{U}_{0},\ldots ,\mathcal{U}%
_{n-1}\right) +q\left( \mathcal{V}_{0},\ldots ,\mathcal{V}_{m-1}\right) \in
\beta _{\mathfrak{F}}S
\end{equation*}%
since $\left( \beta _{\mathfrak{F}}S,+\right) $ is a semigroup. We now want
to prove that $r\left( \mathcal{U}_{0},\ldots ,\mathcal{U}_{n-1},\mathcal{V}%
_{0},\ldots ,\mathcal{V}_{m-1}\right) =f_{r}\left( \mathcal{U}\otimes 
\mathcal{V}\right) $. We denote by $\overline{a}$ and $\overline{b}$ tuples $%
\left( a_{0},\ldots ,a_{n-1}\right) $ and $\left( b_{0},\ldots
,b_{m-1}\right) $ of elements of $S$. Suppose that $C\in S$. We have $C\in
f_{r}\left( \mathcal{U}\otimes \mathcal{V}\right) $ if and only if $\left( 
\mathcal{U}\otimes \mathcal{V}\right) \left( \overline{a},\overline{b}%
\right) $, $f_{r}\left( \overline{a},\overline{b}\right) \in C$, if and only
if $\left( \mathcal{U}\overline{a}\right) $, $\left( \mathcal{V}\overline{b}%
\right) $, $f_{r}\left( \overline{a},\overline{b}\right) \in C$, if and only
if $\left( \mathcal{U}\overline{a}\right) $, $\left( \mathcal{V}\overline{b}%
\right) $, $f_{p}\left( \overline{a}\right) +f_{q}\left( \overline{b}\right)
\in C$, if and only if $\left( f_{p}\left( \mathcal{U}\right) x\right) $, $%
\left( f_{q}\left( \mathcal{V}\right) y\right) $, $x+y\in C$, if and only if
(using the inductive assumption) $\left( p\left( \mathcal{U}_{0},\ldots ,%
\mathcal{U}_{n-1}\right) x\right) $, $\left( q\left( \mathcal{V}_{0},\ldots ,%
\mathcal{V}_{m-1}\right) y\right) $, $x+y\in C$, if and only if $\left(
\left( p\left( \mathcal{U}_{0},\ldots ,\mathcal{U}_{n-1}\right) +q\left( 
\mathcal{V}_{0},\ldots ,\mathcal{V}_{m-1}\right) \right) z\right) $ $z\in C$%
, if and only if $C\in r\left( \mathcal{U}_{0},\ldots ,\mathcal{U}_{n-1},%
\mathcal{V}_{0},\ldots ,\mathcal{V}_{m-1}\right) =p\left( \mathcal{U}%
_{0},\ldots ,\mathcal{U}_{n-1}\right) +q\left( \mathcal{V}_{0},\ldots ,%
\mathcal{V}_{m-1}\right) $. This concludes the proof that $f_{r}\left( 
\mathcal{U}\otimes \mathcal{V}\right) =r\left( \mathcal{U}_{0},\ldots ,%
\mathcal{U}_{n-1},\mathcal{V}_{0},\ldots ,\mathcal{V}_{m-1}\right) $.

Suppose now that the conclusion is true for $p\left( \overline{y}\right) $
and $a\in S$. Let $+\in \mathfrak{S}$ be such that $a+\mathfrak{F}\supset 
\mathfrak{F}$. We prove that the conclusion holds for $q\left( \overline{y}%
\right) :=a+p\left( \overline{y}\right) $. The proof that the conclusion
holds for $p\left( \overline{y}\right) +a$ under the assumption that $%
\mathfrak{F}+a\supset \mathfrak{F}$ is analogous. Suppose that $\mathcal{V}%
_{0},\ldots ,\mathcal{V}_{m-1}\in \beta _{\mathfrak{F}}S$. Then by recursive
assumption we have that $p\left( \mathcal{V}_{0},\ldots ,\mathcal{V}%
_{m-1}\right) \in \beta _{\mathfrak{F}}S$. Therefore $a+p\left( \mathcal{V}%
_{0},\ldots ,\mathcal{V}_{m-1}\right) \supset a+\mathfrak{F}\supset 
\mathfrak{F}$ and hence $a+p\left( \mathcal{V}_{0},\ldots ,\mathcal{V}%
_{m-1}\right) \in \beta _{\mathfrak{F}}S$. Set $\mathcal{V}:=\mathcal{V}%
_{0}\otimes \cdots \otimes \mathcal{V}_{m-1}$. We now prove that $%
f_{q}\left( \mathcal{V}\right) =a+p\left( \mathcal{V}_{0},\ldots ,\mathcal{V}%
_{m-1}\right) $. Suppose that $C$ is a subset of $S$. We have that $C\in
f_{q}\left( \mathcal{V}\right) $ if and only if $\left( \mathcal{V}\overline{%
b}\right) $ $f_{q}\left( \overline{b}\right) \in C$ if and only if $\left( 
\mathcal{V}\overline{b}\right) $ $a+f_{p}\left( \overline{b}\right) \in C$
if and only if $\left( f_{p}\left( \mathcal{V}\right) z\right) $ $a+z\in C$
if and only if (using the inductive hypothesis) $\left( p\left( \mathcal{V}%
_{0},\ldots ,\mathcal{V}_{n-1}\right) z\right) $ $a+z\in C$, if and only if $%
\left( \left( a+p\left( \mathcal{V}_{0},\ldots ,\mathcal{V}_{n-1}\right)
\right) z\right) $ $z\in C$. This shows that $f_{q}\left( \mathcal{V}\right)
=a+p\left( \mathcal{V}_{0},\ldots ,\mathcal{V}_{n-1}\right) =q\left( 
\mathcal{V}_{0},\ldots ,\mathcal{V}_{n-1}\right) $, concluding the proof.
\end{proof}

\subsection{A polynomial Ramsey theorem for actions of trees}

In this section we provide a generalization of the polynomial
Milliken-Taylor theorem of Bergelson, Hindman, and Williams \cite[Corollary
3.5]{bergelson_polynomial_2014} to the setting of actions of trees on
filtered semigroups. In the following we suppose as above that $m,k\in 
\mathbb{N}$ and $\lambda :m\rightarrow k$ is a function. Suppose that $%
\left( S,\mathfrak{F}\right) $ is a filtered set, and $\mathfrak{S}$ is a
collection of semigroup operations as in Subsection \ref{Subsection:extended}%
. Let $p\left( x_{0},\ldots ,x_{m-1}\right) $ be an extended polynomial
defined with respect to $\mathfrak{S}$ and $\mathfrak{F}$. For $i\in k$ we
let $T_{i}$ be a finite tree and $+_{i}$ be an element of $\mathfrak{S}$.
Consider for $i\in k$ a Ramsey action $\alpha _{i}$ of $T_{i}$ on $\left(
S,+_{i},\mathfrak{F}\right) $ in the sense of Definition \ref%
{Definition:action-filtered}. This is given by an order-preserving
assignment \ $T\rightarrow \mathcal{S}\left( S,+_{i},\mathfrak{F}_{i}\right) 
$, $t\mapsto S_{i,t}$, and a semigroup $\mathcal{F}_{\alpha _{i}}$ of
filtered semigroup homomorphisms of $\left( S,+_{i},\mathfrak{F}_{i}\right) $
such that every $\tau \in \mathcal{F}_{\alpha _{i}}$ maps $S_{i,t}$ to $%
S_{i,f_{\tau \left( t\right) }}$, where $f_{\tau }:T_{i}\rightarrow T_{i}$
is a regressive homomorphism. We will furthermore assume that $\mathcal{F}%
_{\alpha _{i}}$ is finite for every $i\in k$.

Let $S_{\alpha _{i}}$ be the set of functions $x:T_{i}\rightarrow S$ such
that $x\left( t\right) \in S_{i,t}$ for every $t\in T_{i}$. Given a subset $%
A $ of $\omega $, a sequence $\left( x_{n}\right) _{n\in A}$ in $S_{\alpha
_{i}}$, and $t\in T_{i}$, we define the $\left( \alpha _{i},+_{i},t\right) $-%
\emph{semigroup} $\mathrm{S}_{\alpha _{i},+_{i},t}\left( x_{n}\right) _{n\in
A}$ generated by $\left( x_{n}\right) _{n\in A}$ to be the collection of
elements of $S_{i,t}$ of the form%
\begin{equation*}
\sum\nolimits_{d\in F}^{i}\tau _{d}\left( x_{d}\left( t_{d}\right) \right)
\end{equation*}%
for finite $F\subset A$, and families $\left( t_{d}\right) _{d\in A}$ in $%
T_{i}$, $\left( \tau _{d}\right) _{d\in A}$ in $\mathcal{F}_{\alpha _{i}}$
such that $\left\{ f_{\tau _{d}}\left( t_{d}\right) :d\in F\right\} $ is a
chain in $T_{i}$ with least element $t$. A sequence $\left( y_{n}\right) $
in $S_{\alpha _{i}}$ is an $\left( \alpha _{i},+_{i}\right) $-\emph{subsystem%
} of $\left( x_{n}\right) $ if $\mathrm{dom}\left( y_{n}\right) <\mathrm{dom}%
\left( y_{n+1}\right) $ and $y_{n}\left( t\right) \in \mathrm{\mathrm{S}}%
_{\alpha _{i},+_{i},t}\left( x_{\ell }\right) _{\ell \in \omega }$ for every 
$n\in \omega $ .

\begin{theorem}
\label{Theorem:polynomial-Gowers}Fix, for every $i\in k$, a sequence $\left(
x_{i,n}\right) _{n\in \omega }$ in $S_{\alpha _{i}}$. Let $A_{\left(
t_{s}\right) _{s\in m}}$ be a subset of $S$ whenever $t_{s}\in T_{\lambda
\left( s\right) }$ for $s\in m$. The following assertions are equivalent:

\begin{enumerate}
\item there exist order-preserving idempotent elements $\xi _{i}\in \left(
\beta _{\mathfrak{F}_{i}}S_{i}\right) _{\alpha _{i}}$ such that\emph{\ }$%
\mathrm{S}_{\alpha _{i},+_{i},t}\left( x_{i,n}\right) _{n\geq r}\in \xi
_{i}\left( t\right) $ for every $i\in k$, $t\in T_{i}$, $r\in \omega $, and $%
A_{\left( t_{s}\right) _{s\in m}}\in p\left( \xi _{\lambda \left( 0\right)
}\left( t_{0}\right) ,\ldots ,\xi _{\lambda \left( m-1\right) }\left(
t_{m-1}\right) \right) $ whenever $t_{s}\in T_{\lambda \left( s\right) }$
for $s\in m$;

\item for any finite coloring $c$ of $S$, for each $i\in k$, sequence $(\psi
_{i,n}^{(\mathfrak{F})})$ of functions $\psi _{i,n}^{(\mathfrak{F})}:\left(
S^{T_{i}}\right) ^{n}\rightarrow \mathfrak{F}_{i}$, there exists an $\left(
\alpha _{i},+_{i}\right) $-tetris subsystem $\left( y_{i,n}\right) $ of $%
\left( x_{i,n}\right) $ such that for every choice of $t_{i}\in T_{i}$ for $%
i\in k$ one has that

\begin{itemize}
\item $y_{i,n}\left( t\right) \in S_{i,t}\cap \psi _{i,n}^{(\mathfrak{F}%
)}\left( x_{i,0},\ldots ,x_{i,n-1}\right) $ for every $t\in T_{i}$;

\item $p\left( \sum\nolimits_{d\in F_{s}}^{\lambda \left( s\right) }\tau
_{\lambda (s),d}\left( y_{\lambda (s),d}\left( t_{\lambda (s),d}\right)
\right) \right) _{s\in m}\in A_{\left( t_{s}\right) _{s\in m}}$ whenever $%
F_{0}<F_{1}<\cdots <F_{m-1}$ are finite subsets of $\omega $, $\left(
t_{i,d}\right) _{d\in \omega }$ is a sequence in $T_{i}$ and $\left( \tau
_{i,d}\right) _{d\in \omega }$ is a sequence in $\mathcal{F}_{\alpha _{i}}$
for $i\in k$ such that $\{f_{\tau _{\lambda \left( s\right) ,d}}\left(
t_{\lambda \left( s\right) ,d}\right) :d\in F_{s}\}$ is a chain in $%
T_{\lambda \left( s\right) }$ with least element $t_{s}$;

\item the color of $\tau \left( y_{i,n}\left( t\right) \right) $ depends
only on $f_{\tau }\left( t\right) $ for every $i\in k$ and $n\in \omega $.
\end{itemize}

\item The same as (2) where moreover $\mathrm{\mathrm{S}}_{\alpha
_{i},+_{i},t}\left( y_{i,n}\right) _{n\in \omega }$ is monochromatic for
every $i\in k$, $t\in T_{i}$.
\end{enumerate}
\end{theorem}

\begin{proof}
The proof of the implication (1)$\Rightarrow $(3) is the same as the proof
of Theorem \ref{Theorem:monochromatic-action}. The implication (3)$%
\Rightarrow $(2) is obvious. We consider the implication (2)$\Rightarrow $%
(1). Define for $i\in k$ and $t\in T_{i}$ the filter $\mathfrak{G}_{i,t}$ on 
$S$ consisting of the sets of the form $\mathrm{S}_{\alpha
_{i},+_{i},t}\left( y_{i,n}\right) _{n\geq r}\cap C$ for $r\in \omega $ and $%
C\in \mathfrak{F}$. Observe that such sets are nonempty by assumption.
Define then $\beta _{\mathfrak{G}_{i,t}}S$ to be the set of ultrafilters on $%
S$ that contain $\mathfrak{G}_{i,t}$. The Ramsey action $\alpha _{i}$ of $%
T_{i}$ on $S$ induces a Ramsey action $\alpha _{i}$ of $T_{i}$ on $\beta S$
given by the map $t\mapsto \beta _{\mathfrak{G}_{i,t}}S_{i,t}$ and the
collection of semigroup homomorphisms $\tau :\beta S\rightarrow \beta S$
given by the continuous extensions of elements $\tau $ of $\mathcal{F}%
_{\alpha _{i}}$. The assumption implies by compactness that the set of
functions $\xi _{i}:T_{i}\rightarrow \beta S$ such that $\xi _{i}\left(
t\right) \in \beta _{\mathfrak{G}_{i,t}}S_{i,t}$ and $\tau \circ \xi
_{i}=\xi _{i}\circ f_{\tau }$ for every $\tau \in \mathcal{F}_{\alpha _{i}}$%
. Hence by Proposition \ref{Proposition:tree} one can find such a function $%
\xi _{i}:T_{i}\rightarrow \beta S$ such that moreover $\xi _{i}\left(
t\right) $ is an idempotent element of $\beta _{\mathfrak{G}_{i,t}}S_{i,t}$
and $\xi _{i}\left( t\right) \leq \xi _{i}\left( t_{0}\right) $ whenever $%
t,t_{0}\in T$ are such that $t\leq t_{0}$. Finally observe that since $\xi
_{i}\left( t\right) \in \beta _{\mathfrak{G}_{i,t}}S_{i,t}$ we have that $%
\mathrm{TS}_{\alpha _{i},+_{i},t}\left( y_{i,n}\right) _{n\geq r}\in \xi
_{i}\left( t\right) $ for every $r\in \omega $. Since $\left( y_{i,n}\right) 
$ is an $\left( \alpha _{i},+_{i}\right) $-tetris subsystem of $\left(
x_{i,n}\right) $, this implies that $\mathrm{TS}_{\alpha _{i},+_{i},t}\left(
x_{i,n}\right) _{n\geq r}\in \xi _{i}\left( t\right) $. This concludes the
proof.
\end{proof}

The polynomial Milliken--Taylor theorem of Bergelson--Hindman--Williams \cite%
[Corollary 3.5]{bergelson_polynomial_2014} is the particular instance of
Theorem \ref{Theorem:polynomial-Gowers} where each tree $T_{i}$ has only one
node, the filter $\mathfrak{F}$ on $S$ is the trivial filter $\left\{
S\right\} $, each semigroup operation $+\in \mathfrak{S}$ on $S$ is
everywhere defined.

\section{Applications to combinatorial configurations in delta sets\label%
{Section:combinatorial}}

\subsection{Configurations in delta sets in amenable groups}

Suppose that $\left( G,\cdot \right) $ is a discrete \emph{amenable} group.
Recall that $G$ is endowed with a canonical notion of \emph{density}---also
known as Banach density---defined as follows. Let $A$ be a subset of $G$,
and $\alpha $ be a positive real number. Then the density $d(A)$ of $A$ is
larger than or equal to $\alpha $ if and only if for any finite subset $F$
of $G$ and any $\varepsilon >0$ there exists a finite subset $I$ of $G$ such
that $\left\vert A\cap I\right\vert \geq \left( \alpha -\varepsilon \right)
\left\vert I\right\vert $ and $\left\vert gI\bigtriangleup I\right\vert
<\varepsilon \left\vert I\right\vert $ for every $g\in F$. A subset $A$ of $%
G $ has positive density if $d(A)>0$. We recall for future use the following 
\emph{correspondence principle}, which is established in \cite[Lemma 4.6]%
{beiglbock_sumset_2010}.

\begin{lemma}
\label{Lemma:correspondence}Suppose that $A\subset G$ has positive density.
Then there exist a standard Borel probability space $\left( X,\mu \right) $,
a measure-preserving action $g\mapsto \tilde{g}\in \mathrm{Aut}\left( X,\mu
\right) $ of $G$ on $\left( X,\mu \right) $, and a Borel subset $\tilde{A}$
of $X$ such that $\mu (\tilde{A})=d(A)$, and for any finite subset $F$ of $G$%
, one has that%
\begin{equation*}
d\left( \bigcap_{g\in F}gA\right) \geq \mu \left( \bigcap_{g\in F}\tilde{g}%
\tilde{A}\right) \text{.}
\end{equation*}
\end{lemma}

The following result about sets of positive density was proved by
Furstenberg as an application of the corresponding principle and Hindman's
theorem on finite unions \cite{furstenberg_recurrence_1981}. Recall that the
delta set $AA^{-1}$ associated with a subset $A$ of $G$ is the set $\left\{
gh^{-1}:g,h\in A\right\} \subset G$.

\begin{theorem}[Furstenberg]
\label{Theorem:Furstenberg}Suppose that $A\subset G$ has positive density,
and $\left( g_{n}\right) $ is a sequence of elements of $G$. Then there
exists a sequence $\left( F_{n}\right) $ of finite nonempty subsets of $%
\omega $ such that $F_{n}<F_{n+1}$ for every $n\in \omega $ and such that,
for every finite subset $E$ of $\omega $ one has that%
\begin{equation*}
\prod_{d\in E}\prod_{i\in F_{d}}g_{i}\in AA^{-1}\text{.}
\end{equation*}
\end{theorem}

In the statement of Theorem \ref{Theorem:Furstenberg} and in the following,
we convene that, for a given finite nonempty subet $F$ of $\omega $ and
family $\left( g_{i}\right) _{i\in F}$ in $G$, the product $\prod_{i\in
F}g_{i}$ denotes the element $g_{i_{0}}\cdots g_{i_{k-1}}$ of $G$, where $%
\left( i_{s}\right) _{s\in k}$ is the increasing enumeration of $F$. If $F$
is empty, then $\prod_{i\in F}g_{i}$ denotes the identity of $G$.

Fix now a finite rooted tree $T$. Let $T^{+}$ be the set of nodes in $T$
different from the root. Recall that $\mathrm{FIN}^{T}$ denotes the adequate
partial semigroup of functions $b:\mathrm{dom}\left( b\right) \rightarrow
T^{+}$, where $\mathrm{dom}\left( b\right) $ is a (possibly empty) finite
subset of $\omega $. If $f:T\rightarrow T$ is a regressive
homomorphism---see Definition \ref{Definition:regrssive}---then $\tau _{f}:%
\mathrm{FIN}^{T}\rightarrow \mathrm{FIN}^{T}$ is the corresponding adequate
partial semigroup homomorphism defined by $\tau _{f}\left( b\right) :\left\{
n\in \mathrm{dom}\left( b\right) :f\left( b\left( n\right) \right) \in
T^{+}\right\} \rightarrow T^{+}$, $n\mapsto f\left( b\left( n\right) \right) 
$. For $t\in T^{+}$ we let $\mathrm{FIN}_{t}^{T}$ be the set of $b\in 
\mathrm{FIN}^{T}$ such that the range of $b$ is a nonempty chain in $T$ with
least element $t$. If $r$ is the root of $T$, then $\mathrm{FIN}_{r}^{T}$
contains only the empty function.

\begin{theorem}
\label{Theorem:Gowers-Furstenberg}Suppose that $A\subset G$ has positive
density. Let $\left( g_{n}\right) $ be a sequence of functions $%
g:T^{+}\rightarrow G$. There exists a sequence of functions $%
b_{n}:T\rightarrow \mathrm{FIN}^{T}$, $t\mapsto b_{n,t}\in \mathrm{FIN}%
_{t}^{T}$ such that $\mathrm{\mathrm{dom}}\left( b_{n,t}\right) <\mathrm{dom}%
\left( b_{n+1,t}\right) $ and for any finite nonempty subset $F$ of $\omega $%
, regressive homomorphisms $f_{d}:T\rightarrow T$ and nodes $t_{d}\in T$ for 
$d\in F$ such that $\left\{ f_{d}\left( t_{d}\right) :d\in F\right\} $ is a
chain in $T$, one has that%
\begin{equation*}
\prod_{d\in F}\prod_{\substack{ j\in \mathrm{dom}(b_{d,t_{d}})  \\ %
(f_{d}\circ b_{d,t_{d}})\left( j\right) \in T^{+}}}\left( g_{j}\circ
f_{d}\circ b_{d,t_{d}}\right) \left( j\right) \in AA^{-1}\text{.}
\end{equation*}
\end{theorem}

Theorem \ref{Theorem:Furstenberg} is the particular instance of \ref%
{Theorem:Gowers-Furstenberg} where $T$ is the rooted tree with only two
nodes.

Before proving Theorem \ref{Theorem:Gowers-Furstenberg} we recall some
notation about ultrafilter limits. Let $\mathcal{U}$ be an ultrafilter on a
set $S$. Fix a topological space space $X$, a function $f:S\rightarrow X$,
and $x\in X$. Then we say that $x$ is the $\mathcal{U}$-limit of $f$, in
formulas $\mathcal{U}-\lim_{s}f\left( s\right) =x$, if for every
neighborhood $W$ of $x$ in $X$, $\mathcal{U}s$, $f\left( s\right) \in W$. It
follows from the properties of ultrafilters that, if $X$ is a compact
Hausdorff space, then any function $f:A\rightarrow X$ admits a unique $%
\mathcal{U}$-limit, which we also denote by $f\left( \mathcal{U}\right) $.

Suppose that $H$ is a Hilbert space. Then the set $\mathrm{Ball}(B(H))$ of
bounded linear operators on $H$ of norm at most $1$ is a compact semigroup
with respect to the composition operation and the weak operator topology 
\cite[Section I.3.1]{blackadar_operator_2006}. The group $U(H)$ of unitary
operators on $H$ is a subsemigroup of $\mathrm{Ball}\left( B(H)\right) $.
Suppose now that $S$ is a partial semigroup, and $v:S\rightarrow \mathrm{Ball%
}\left( B(H)\right) $, $s\mapsto v_{s}$ is a partial semigroup homomorphism.
If $\mathcal{U}\in \gamma S$ is idempotent, then $v_{\mathcal{U}}=\mathcal{U}%
-\lim_{s}v_{s}$ is an orthogonal projection. Indeed, since $\mathcal{U}$ is
idempotent and $v:\gamma S\rightarrow \mathrm{Ball}\left( B(H)\right) $ is a
semigroup homomorphism, $v_{\mathcal{U}}$ is idempotent. Since $\left\Vert
v_{\mathcal{U}}\right\Vert \leq 1$, $v_{\mathcal{U}}$ must be an orthogonal
projection. As in the proof of \cite[Corollary 2.1]{bergelson_ip-sets_1996}
one can deduce from this observation the following.

\begin{lemma}
\label{Lemma:idempotent-measure}Let $\left( X,\mu \right) $ be a probability
space, and $\tilde{A}$ be a measurable subset of $X$. Let $\mathrm{Aut}%
\left( X,\mu \right) $ be the group of invertible measure-preserving
transformations of $\left( X,\mu \right) $. Suppose that $S$ is an adequate
partial semigroup, $S\rightarrow \mathrm{\mathrm{Au}t}\left( X,\mu \right) $%
, $s\mapsto \alpha _{s}$ is a partial semigroup homomorphism, and $\mathcal{U%
}\in \gamma S$ is idempotent. Then%
\begin{equation*}
\mathcal{U}-\lim_{s}\mu \left( \tilde{A}\cap \alpha _{s}(\tilde{A})\right)
\geq \mu (\tilde{A})^{2}\text{.}
\end{equation*}
\end{lemma}

\begin{proof}
Let $H:=L^{2}\left( X,\mu \right) $. For every $s\in S$ define $v_{s}\in
U(H) $ by $v_{s}f=f\circ \alpha _{s}^{-1}$ for $f\in H$. This defines a
partial semigroup homomorphism $S\rightarrow \mathrm{Ball}\left( B(H)\right) 
$, $s\mapsto v_{s}$. As observed above, $v_{\mathcal{U}}$ is an orthogonal
projection. Let $\xi $ be the vector of $H$ corresponding to the
characteristic function of $\tilde{A}$, and $\xi _{0}$ be the vector of $H$
corresponding to the function constantly equal to $1$. Observe that $%
v_{s}\xi _{0}=\xi _{0}$ for every $s\in S$. Therefore $\left\langle v_{%
\mathcal{U}}\xi _{0},\xi _{0}\right\rangle =\left\Vert \xi _{0}\right\Vert
^{2}$ and hence $v_{\mathcal{U}}\xi _{0}=\xi _{0}$. Then we have that%
\begin{equation*}
\mathcal{U}-\lim_{s}\mu \left( \tilde{A}\cap \alpha _{s}(\tilde{A})\right) =%
\mathcal{U}-\lim_{s}\left\langle \xi ,v_{s}\xi \right\rangle =\left\langle
\xi ,v_{\mathcal{U}}\xi \right\rangle =\left\Vert v_{\mathcal{U}}\xi
\right\Vert ^{2}\geq \left\langle \nu _{\mathcal{U}}\xi ,\xi
_{0}\right\rangle ^{2}=\left\langle \xi ,\xi _{0}\right\rangle ^{2}=\mu
(A)^{2}\text{.}
\end{equation*}
\end{proof}

We will now prove Theorem \ref{Theorem:Gowers-Furstenberg} using Lemma \ref%
{Lemma:idempotent-measure}.

\begin{proof}[Proof of Theorem \protect\ref{Theorem:Gowers-Furstenberg}]
Recall that the collection of adequate partial semigroup homomorphisms $\tau
_{f}:\mathrm{FIN}^{T}\rightarrow \mathrm{FIN}^{T}$ when $f$ ranges among the
regressive homomorphisms of $T$, and the collection of adequate partial
subsemigroups $\mathrm{FIN}_{t}^{T}$ for $t\in T$, define a Ramsey action of 
$T$ on $\mathrm{FIN}^{T}$; see Definition \ref{Definition:Ramsey-action}.
Therefore by Theorem \ref{Theorem:monochromatic-tree-action} and Proposition %
\ref{Proposition:tree} there exists a function $\xi :T\rightarrow \gamma 
\mathrm{FIN}^{T}$ such that $\xi \left( t\right) \in \gamma \mathrm{FIN}%
_{t}^{T}$ is idempotent, $\xi \left( t\right) \leq \xi \left( t_{0}\right) $
for every $t,t_{0}\in T$ with $t\leq t_{0}$, and $\tau _{f}\circ \xi =\xi
\circ f$ for any regressive homomorphism $f:T\rightarrow T$.

Let $\left( X,\mu \right) $ and $\tilde{A}\subset X$ be, respectively, the
standard probability space and the Borel set obtained from $A$ applying
Lemma \ref{Lemma:correspondence}. For $b\in \mathrm{FIN}^{T}$ define%
\begin{equation*}
g_{b}:=\prod_{i\in \mathrm{dom}\left( b\right) }g_{i}\left( b\left( i\right)
\right) \in G\text{.}
\end{equation*}%
This gives a partial semigroup homomorphism $\mathrm{FIN}^{T}\rightarrow G$, 
$b\mapsto g_{b}$. In turns, this gives a partial semigroup homomorphism $%
\mathrm{FIN}^{T}\rightarrow \mathrm{\mathrm{\mathrm{Aut}}}\left( X,\mu
\right) $, $b\mapsto \tilde{g}_{b}$. Therefore by Lemma \ref%
{Lemma:idempotent-measure} we have that, for every $t\in T$, 
\begin{equation*}
d(A)^{2}=\mu (\tilde{A})^{2}\leq \xi \left( t\right) -\lim_{b\in \mathrm{FIN}%
_{t}^{T}}\mu \left( \tilde{A}\cap \tilde{g}_{b}\tilde{A}\right) \leq \xi
\left( t\right) -\lim_{b\in \mathrm{FIN}_{t}^{T}}d\left( A\cap g_{b}A\right) 
\text{.}
\end{equation*}%
Observe that, whenever $d\left( A\cap gA\right) >0$ one has that $A\cap
gA\neq \varnothing $ and hence $g\in AA^{-1}$. Since $d(A)>0$, we deduce in
particular that, for every $t\in T$, $\xi \left( t\right) b$, $g_{b}\in
AA^{-1}$. Observe now that the element 
\begin{equation*}
\prod_{d\in F}\prod_{\substack{ j\in \mathrm{dom}(b_{d,t_{d}})  \\ %
(f_{d}\circ b_{d,t_{d}})\left( j\right) \in T^{+}}}\left( g_{j}\circ
f_{d}\circ b_{d,t_{d}}\right) \left( j\right)
\end{equation*}%
of $G$ in the statement can be written, in the notation above, as 
\begin{equation*}
g_{\sum_{d\in F}\tau _{f_{d}}(b_{d,t_{d}})}\text{.}
\end{equation*}%
Therefore the desired conclusion follows from Theorem \ref%
{Theorem:monochromatic-action}.
\end{proof}

The same argument as the one in the proof of Theorem \ref%
{Theorem:Gowers-Furstenberg} gives the following generalization. Suppose as
before that $\left( g_{n}\right) $ is a sequence of functions $%
g_{n}:T^{+}\rightarrow G$. Let also $L$ be a subset of $G$. Adopting the
notation from Subsection \ref{Subsection:common}, we let $\mathrm{FIN}%
_{L}^{T}$ be the set of finitely-supported functions $b:\mathrm{dom}\left(
b\right) \rightarrow L\cup T^{+}$, where $\mathrm{dom}\left( b\right) $ is a
(possibly empty) finite subset of $\omega $. For $t\in T^{+}$ we let $%
\mathrm{FIN}_{L,t}^{T}$ be the set of $b\in \mathrm{FIN}_{L}^{T}$ such that
the intersection of the range of $b$ with $T$ is a nonempty chain with least
element $t$. We also let $\mathrm{FIN}_{L,r}^{T}$, where $r$ is the root of $%
T$, be the set of $b\in \mathrm{FIN}^{T}$ such that the range of $b$ is
contained in $L$. Recall the notion of variable subtitution map $\sigma :%
\mathrm{dom}\left( \sigma \right) \subset L\cup T^{+}\rightarrow L\cup T^{+}$
with corresponding regressive homomorphism $f_{\sigma }:T\rightarrow T$ as
in Definition \ref{Definition:variable-substitution}. Any such map defines
an adequate partial semigroup homomorphism $\tau _{\sigma }:\mathrm{FIN}%
_{L}^{T}\rightarrow \mathrm{FIN}_{L}^{T}$ obtained by setting, for $b\in 
\mathrm{FIN}_{L}^{T}$, $\tau _{\sigma }\left( b\right) :\left\{ n\in \mathrm{%
dom}\left( b\right) :\sigma \left( b\left( n\right) \right) \in \mathrm{%
\mathrm{\mathrm{\mathrm{\mathrm{dom}}}}}\left( \sigma \right) \right\}
\rightarrow L\cup T^{+}$, $n\mapsto \sigma \left( b\left( n\right) \right) $.

\begin{theorem}
\label{Theorem:Gowers-Hales-Jewett-Furstenberg}Suppose that $A\subset G$ has
positive density, and $L\subset G$. Let $\left( g_{n}\right) $ be a sequence
of functions $g_{n}:T^{+}\rightarrow G$. Extend $g_{n}$ to a function $%
g_{n}:T^{+}\cup L\rightarrow G$ by setting $g_{n}\left( x\right) =x$. For
any sequence $\left( \mathcal{F}_{n}\right) $ of finite sets of variable
substitution maps for $L,T$, there exists a sequence of functions $%
b_{n}:T\rightarrow \mathrm{FIN}_{L}^{T}$, $t\mapsto b_{n,t}\in \mathrm{FIN}%
_{L,t}^{T}$ such that $\mathrm{\mathrm{dom}}\left( b_{n,t}\right) <\mathrm{%
dom}\left( b_{n+1,t}\right) $ and for any finite nonempty subset $F$ of $%
\omega $, variable substitution maps $\sigma _{d}\in \mathcal{F}_{d}$ and
nodes $t_{d}\in T$ for $d\in F$ such that $\left\{ f_{\sigma _{d}}\left(
t_{d}\right) :d\in F\right\} $ is a chain in $T$, one has that%
\begin{equation*}
\prod_{d\in F}\prod_{\substack{ j\in \mathrm{dom}(b_{d,t_{d}})  \\ %
b_{d,t_{d}}\left( j\right) \in \mathrm{dom}\left( \sigma _{d}\right) }}%
\left( g_{j}\circ \sigma _{d}\circ b_{d,t_{d}}\right) \left( j\right) \in
AA^{-1}\text{.}
\end{equation*}
\end{theorem}

\subsection{Polynomial configurations in delta sets in $\mathbb{Z}^{m}$}

We now specialize the discussion in the case when $G$ is the (additive)
group $\mathbb{Z}^{m}$ for some $m\in \mathbb{N}$. In this setting one can
obtain a \emph{polynomial }strengthening of Theorem \ref%
{Theorem:Gowers-Furstenberg}. Denote by $\mathrm{Int}[z_{0},\ldots ,z_{m-1}]$
the space of polynomial functions $p:\mathbb{Z}^{d}\rightarrow \mathbb{Z}$
that are defined polynomials with rational coefficients in the variables $%
z_{0},\ldots ,z_{m-1}$ and which map integer points to integer points.

\begin{theorem}
\label{Theorem:polynomial-Furstenberg-Gowers}Suppose $p=\left( p_{0},\ldots
,p_{d-1}\right) \in \mathrm{Int}[z_{0},\ldots ,z_{m-1}]^{d}$. We will regard 
$p$ as a function $p:\mathbb{Z}^{m}\rightarrow \mathbb{Z}^{d}$, and assume
that $p\left( 0\right) =0$. Suppose that $\left( g_{n}\right) $ is a
sequence of functions $g_{n}:T^{+}\rightarrow \mathbb{Z}^{m}$, and $A\subset 
\mathbb{Z}^{d}$ is a subset of positive density. There exists a sequence of
functions $b_{n}:T\rightarrow \mathrm{FIN}^{T}$, $t\mapsto b_{n,t}\in 
\mathrm{FIN}_{t}^{T}$ such that $\mathrm{\mathrm{dom}}\left( b_{n,t}\right) <%
\mathrm{dom}\left( b_{n+1,t}\right) $ and for any finite nonempty subset $F$
of $\omega $, regressive homomorphisms $f_{d}:T\rightarrow T$ and nodes $%
t_{d}\in T$ for $d\in F$ such that $\left\{ f_{d}\left( t_{d}\right) :d\in
F\right\} $ is a chain in $T$, one has that%
\begin{equation*}
p\left( \sum_{d\in F}\sum_{\substack{ j\in \mathrm{dom}(b_{d,t_{d}}) \\ %
(f_{d}\circ b_{d,t_{d}})\left( j\right) \in T^{+}}}\left( g_{j}\circ
f_{d}\circ b_{d,t_{d}}\right) \left( j\right) \right) \in A-A\text{.}
\end{equation*}
\end{theorem}

Theorem \ref{Theorem:polynomial-Furstenberg-Gowers} is a common
generalization of Theorem \ref{Theorem:Gowers-Furstenberg} and \cite[%
Corollary 2.2]{bergelson_ip-sets_1996}. The proof of Theorem \ref%
{Theorem:polynomial-Furstenberg-Gowers} is analogous to the proof of Theorem %
\ref{Theorem:Gowers-Furstenberg}, where Lemma \ref{Lemma:idempotent-measure}
is replaced by the following lemma, proved in \cite[Theorem 6.1]%
{schnell_idempotent_2007}.

\begin{lemma}
Let $\left( X,\mu \right) $ be a probability space, and $\tilde{A}$ be a
measurable subset of $X$. Let $T_{0},\ldots ,T_{d-1}$ be pairwise commuting
invertible measure-preserving transformations of $\left( X,\mu \right) $.
Suppose $p_{0},\ldots ,p_{d-1}\in \mathrm{Int}[z_{0},\ldots ,z_{m-1}]$ are
such that $p_{i}\left( 0\right) =0$ for $i\in d$. Suppose that $\mathcal{U}%
\in \beta \mathbb{Z}^{m}$ is an idempotent ultrafilter. Then%
\begin{equation*}
\mathcal{U}-\lim_{a\in \mathbb{Z}^{m}}\mu \left( \tilde{A}\cap \left(
\prod_{i\in d}T_{i}^{p_{i}\left( a\right) }\right) (\tilde{A})\right) \geq
\mu (\tilde{A})^{2}\text{.}
\end{equation*}
\end{lemma}

We now present a proof of Theorem \ref{Theorem:polynomial-Furstenberg-Gowers}%
.

\begin{proof}[Proof of Theorem \protect\ref%
{Theorem:polynomial-Furstenberg-Gowers}]
Let the functions $\mathrm{FIN}^{T}\rightarrow \mathbb{Z}^{m}$, $b\mapsto
g_{b}$, and $\xi :T\rightarrow \gamma \mathrm{FIN}^{T}$ be defined as in the
proof of Theorem \ref{Theorem:Gowers-Furstenberg}. Let also $\left( X,\mu
\right) $ and $\tilde{A}$ be obtained from $A$ as in Lemma \ref%
{Lemma:correspondence}. The function $\mathrm{FIN}^{T}\rightarrow \mathbb{Z}%
^{m}$, $b\mapsto g_{b}$, extends to a continuous semigroup homomorphism $%
\gamma \mathrm{FIN}^{T}\rightarrow \beta \mathbb{Z}^{m}$, $\mathcal{V}%
\mapsto g_{\mathcal{V}}$. Set $\mathcal{U}_{t}:=g_{\xi \left( t\right) }$
for $t\in T$, and observe that $\mathcal{U}_{t}\in \beta \mathbb{Z}^{m}$ is
an idempotent ultrafilter. Let also $e_{0},\ldots ,e_{d-1}$ be the canonical
generators of $\mathbb{Z}^{d}$, and $\tilde{e}_{0},\ldots ,\tilde{e}_{d-1}$
be the corresponding elements of $\mathrm{Aut}\left( X,\mu \right) $.
Observe that%
\begin{equation*}
\prod_{i\in d}\tilde{e}_{i}^{p_{i}\left( a\right) }=\prod_{i\in d}\widetilde{%
p_{i}\left( a\right) }=\widetilde{p\left( a\right) }\text{.}
\end{equation*}%
Therefore we have that, for every $t\in T$,%
\begin{equation*}
d(A)^{2}\leq \mathcal{U}_{t}-\lim_{a\in \mathbb{Z}^{m}}\mu \left( \tilde{A}%
\cap \left( \prod_{i\in d}\tilde{e}_{i}^{p_{i}\left( a\right) }\right) 
\tilde{A}\right) =\lim_{b\in \mathrm{FIN}^{T}}\mu \left( \tilde{A}\cap 
\widetilde{p\left( g_{b}\right) }\tilde{A}\right) \leq d\left( A\cap
(p\left( g_{b}\right) +A)\right) \text{.}
\end{equation*}%
Therefore $\xi \left( t\right) b$, $p\left( g_{b}\right) \in A-A$. The
conclusion follows again from Theorem \ref{Theorem:monochromatic-action}.
\end{proof}

In a similar fashion, one can prove the following more general result.

\begin{theorem}
Suppose $p=\left( p_{0},\ldots ,p_{d-1}\right) \in \mathrm{Int}[z_{0},\ldots
,z_{m-1}]^{d}$ is such that $p\left( 0\right) =0$. Suppose that $\left(
g_{n}\right) $ is a sequence of functions $g_{n}:T^{+}\rightarrow \mathbb{Z}%
^{m}$, and $A\subset \mathbb{Z}^{d}$ is a subset of positive density. Fix
also a subset $L$ of $\mathbb{Z}^{d}$. Let $\left( \mathcal{F}_{n}\right) $
be a sequence of variable substitution maps for $T,L$. There exists a
sequence of functions $b_{n}:T\rightarrow \mathrm{FIN}_{L}^{T}$, $t\mapsto
b_{n,t}\in \mathrm{FIN}_{L,t}^{T}$ such that $\mathrm{\mathrm{dom}}\left(
b_{n,t}\right) <\mathrm{dom}\left( b_{n+1,t}\right) $ and for any finite
nonempty subset $F$ of $\omega $, $\sigma _{d}\in \mathcal{F}_{d}$ and nodes 
$t_{d}\in T$ for $d\in F$ such that $\left\{ f_{\sigma _{d}}\left(
t_{d}\right) :d\in F\right\} $ is a chain in $T$, one has that%
\begin{equation*}
p\left( \sum_{d\in F}\sum_{\substack{ j\in \mathrm{dom}(b_{d,t_{d}})  \\ %
b_{d,t_{d}}\left( j\right) \in \mathrm{dom}\left( \sigma _{d}\right) }}%
\left( g_{j}\circ \sigma _{d}\circ b_{d,t_{d}}\right) \left( j\right)
\right) \in A-A.
\end{equation*}
\end{theorem}

\bibliographystyle{amsplain}
\bibliography{bibliography}

\end{document}